\documentclass[12pt]{article}
\usepackage{amssymb,a4}
\usepackage{amsmath,amsfonts,amssymb,amsthm}
\usepackage{mathrsfs}
\newcommand{\al}{\alpha}

\def\wt{{\rm wt}}

\def\de{\delta}

\def\bft{{\bf t}}
\def\C{{\mathbb C}}
\def\Q{{\mathbb Q}}

\def\Z{{\mathbb Z}}

\def\1{{\bf 1}}
\def\l{\lambda}
\def \wt{{\rm wt}}

\def \<{\langle}
\def \>{\rangle}

\def \wg{\widehat{\frak g}}
\def \bft{{\bf t}}

\def \pf{\noindent {\em Proof \, \,}}
\def\theequation{5.\arabic{equation}}

\renewcommand{\theequation}{\thesection.\arabic{equation}}

\newtheorem{theorem}{Theorem}[section]
\newtheorem{prop}[theorem]{Proposition}
\newtheorem{lem}[theorem]{Lemma}
\newtheorem{cor}[theorem]{Corollary}
\newtheorem{remark}[theorem]{Remark}
\theoremstyle{definition}
\newtheorem{defn}[theorem]{Definition}

\begin{document}
\begin{center}
{\Large {\bf The commutant of $L_{\widehat{\frak{sl}}_{2}}(n,0)$ in the vertex operator algebra $L_{\widehat{\frak{sl}}_{2}}(1,0)^{\otimes n}$ }} \\

\vspace{0.5cm} Cuipo(Cuibo) Jiang\footnote{Supported by China NSF grants 10931006, 11371245,  China RFDP grant 2010007310052, and the Innovation Program of Shanghai Municipal Education Commission (11ZZ18)}
\\
Department of Mathematics,  Shanghai Jiaotong University, Shanghai, 200240, China \\
email: cpjiang@sjtu.edu.cn\\

\vspace{.2 cm} Zongzhu Lin\\
 Department of Mathematics, Kansas State University, Manhattan, KS 66506, USA\\
email: zlin@math.ksu.edu
\end{center}
\hspace{1cm}

\begin{abstract}

We study the commutant $L_{\widehat{\frak{sl}}_{2}}(n,0)^c$ of $L_{\widehat{\frak{sl}}_{2}}(n,0)$ in the vertex operator algebra $L_{\widehat{\frak{sl}}_{2}}(1,0)^{\otimes n}$, for $n\geq 2$. The main results include a complete classification of all irreducible  $L_{\widehat{\frak{sl}}_{2}}(n,0)^c$-modules and  that $L_{\widehat{\frak{sl}}_{2}}(n,0)^c$ is a rational vertex operator algebra. As a consequence, every  irreducible $L_{\widehat{\frak{sl}}_{2}}(n,0)^c$-module arises from the coset construction as conjectured in \cite{LS}.

 2000MSC:17B69

\end{abstract}

\section{Introduction}

The Goddard-Kent-Olive coset construction is a construction of certain irreducible lowest weight unitary modules for the
Virasoro algebra using standard modules for affine Lie algebras \cite{GKO1}, \cite{GKO2}. The coset construction was later generalized in \cite{FZ}. It has been  proved in \cite{FZ} that under suitable conditions the centralizer  of a vertex operator algebra  is  a vertex operator subalgebra.  Coset constructions have been used to give many important conformal field theories associated with Kac-Moondy affine algebras(see for examples \cite{BEHHH}, \cite{DL}, \cite{GZ},  \cite{KR}, \cite{DLY2}, \cite{DLWY}, \cite{DW1}, \cite{DW2}, \cite{ALY}, \cite{X1}, \cite{X2}, etc).

Let ${\frak g}$ be a simple finite-dimensional Lie algebra over $\C$ and $\widehat{\frak g}$ the affine Kac-Moondy  Lie algebra associated with ${\frak g}$. For a positive integer $l$, let $L_{\widehat{\frak g}}(l,0)$ be the simple vertex operator algebra corresponding to $\widehat{\frak g}$ and $l$. Then   the  vertex operator algebra $L_{\widehat{\frak g}}(l,0)$ is regular \cite{DLM1}, \cite{LL}. It is well known that $L_{\widehat{\frak g}}(l,0)$ plays a very important role in the theory of vertex operator algebras. For $n\geq 2$, $n\in\Z_{+}$, let $V=L_{\widehat{\frak g}}(l,0)^{\otimes n}$ be  tensor powers of $L_{\widehat{\frak g}}(l,0)$. Then $V$ is still a regular vertex operator algebra and $L_{\widehat{\frak g}}(nl,0)$ can be naturally regarded as a vertex operator subalgebra of $V$ \cite{GKO2}, \cite{LL}, \cite{K}. The centralizer  (also called commutant) $L_{\widehat{\frak g}}(nl,0)^c$ of $L_{\widehat{\frak g}}(nl,0)$ in $V$   is also  called the coset vertex operator algebra associated to the pair $(V, L_{\widehat{\frak g}}(l,0))$. Although rational affine vertex operator algebras and their representations have been well understood, the coset vertex operator algebra $L_{\widehat{\frak g}}(nl,0)^c$ is usually very complicated and little has been known. In this paper we focus on the case ${\frak g}=\frak{sl}_{2}(\C)$ and $l=1$.

By the Goddard-Kent-Olive coset construction,  $L_{\widehat{\frak{sl}}_{2}}(n+1,0)^c$  for $n\geq 1$ contains a
rational vertex operator algebra $L=\otimes_{i=1}^{n}L(c_i, 0)$ and the direct sum decomposition of $L_{\widehat{\frak{sl}}_{2}}(n+1,0)^c$ into irreducible $L$-submodules are known explicitly  (see also \cite{LY}, \cite{LS}). Here $c_i=1-\frac{6}{(i+2)(i+3)}$ and $L(c_i,0)$ is the rational simple Virasoro vertex operator algebra with central charge $c_i$, $1\leq i\leq n$. For simplicity, we set $M^{(n)}=L_{\widehat{\frak{sl}}_{2}}(n+1,0)^c$  in this paper. For $n=2$, the commutant $M^{(2)}$ is one of vertex operator algebras studied in \cite{CL} and its structure, representations and ratonality have been established completely. It was proved in \cite{KLY} ( see also \cite{KMY} and \cite{DLMN}) that  the centralizer $M^{(3)}$ is isomorphic to the simple vertex operator algebra $V_{\sqrt{2}A_{2}}^+$, whose irreducible modules have been classified in \cite{ADL}. While the rationality of $M^{(3)}$ follows from a general result obtained in \cite{DJL}. For $n\geq 4$, the structure of $M^{(n)}$ remains elusive. Based on the work of \cite{LY},
the commutant $M^{(n)}$ was studied in \cite{LS} regarded as the commutant subalgebra of the parafermion vertex operator subalgebra inside $V_{\sqrt{2}A_{n}}$. It was shown in \cite{LS} that $M^{(n)}$ is generated by its weight two subspace.
Irreducible modules from the coset construction have  also been  studied deeply in \cite{LS}. It was conjectured in \cite{LS} that each irreducible $M^{(n)}$-modules should arise from the coset construction. The main goal of this paper is to classify all irreducible $M^{(n)}$-modules (Theorem~\ref{th4.1}) and then use the classification to prove that $M^{(n)}$ is a rational vertex operator algebra for all $n$ (Theorem~\ref{th6.1}). As a consequence the above conjecture is true.

 To prove the main results, the Zhu algebra $A(M^{(n)})$ of $M^{(n)}$ plays an important role. By first giving a  spanning set of $M^{(n)}$, we prove that the Zhu algebra  $A(M^{(n)})$ is generated by  the elements in weight two subspace.  To classify the irreducible  modules of $M^{(n)}$, we use the structure of the Zhu algebra $A(M^{(n)})$ and the known results on $M^{(2)}$ and $M^{(3)}$ to show that irreducible modules of
$M^{(n)}$ can be divided into two types. We find a one-to-one correspondence between the first type irreducible modules of $M^{(n)}$  and the irreducible modules of  a quotient of the group algebra $\C[S_{n+1}]$ corresponding to partitions of $n+1$ of at most two parts. They are also in one-to-one correspondence with irreducible representations of certain Schur algebra $S(2, n+1)$.  The structures of the second type irreducible modules of $M^{(n)}$ are a little more complicated. We have to do  more technical analysis.

It is widely believed that if both  a vertex operator algebra $V$  and its subalgebra $U$ are rational, then the commutant $U^c$ of $U$ in $V$ is rational.  In this paper we prove that this is true for  $L_{\widehat{\frak{sl}}_{2}}(n+1,0)^c=M^{(n)}$.  We achieve this by two steps. The first step is to show that the Zhu algebra $A(M^{(n)})$ is a finite-dimensional semi-simple associative algebra. The proof uses generating sets and conditions of $A(M^{(n)})$ established in Section 4 and  the established isomorphisms from some quotients of $A(M^{(n)})$ to some quotients of the semi-simple associative algebra $\C[S_{n+1}]$ (cf. Lemma~\ref{l4.3}) to prove that every admissible $M^{(n)}$-module is completely reducible.   By the semi-simplicity of $A(M^{(n)})$, if two irreducible modules of $M^{(n)}$ have the same lowest weight,  the extension of one module by the other is trivial. The next step is to show that ${\rm Ext}_{M^{(n)}}(M^2,M^1)=0$ for any two irreducible admissible $M^{(n)}$-modules $M^1$ and $M^2$.
 Thus  every admissible module is semi-simple (cf. \cite{A}).

To compute the extensions, we  choose a suitable  rational vertex operator subalgebra $U$ of $M^{(n)}$ with the same Virasoro vector. Using the fusion rule of rational Virasoro vertex operator algebras and inductive assumption, we prove that  for any $U$-submodules  $N, N^1, N^2$ of  $M^{(n)}, M^1, M^2$ respectively, the fusion rule $I_{U}\left(\begin{tabular}{c}
$N^1$\\
$N$\  $N^2$
\end{tabular}\right)=0$. Then it follows from a lemma given in \cite{DJL} that  ${\rm Ext}_{M^{(n)}}(M^2,M^1)=0$. We expect that our results and methods used in this paper can be some help in studying $L_{\widehat{\frak{g}}}(1,0)^{\otimes(n+1)}$ for general finite-dimensional Lie algebra $\frak{g}$.

This paper is organized as follows: In Section 2, we recall some basic notations and facts on vertex operator algebras. In Section 3, we first list some known results on the commutant $M^{(n)}$ obtained in \cite{LS} and \cite{LY}. We next construct a spanning set (Lemma~\ref{l4.10})  and generators of the Zhu algebra of $M^{(n)}$ (Theorem~\ref{p4.2}). Section 4 is dedicated to the classification of irreducible modules of $M^{(n)}$. The rationality of $M^{(n)}$ is established in Section 5.

\section{Preliminaries}
\def\theequation{2.\arabic{equation}}
\setcounter{equation}{0}

Let $V=(V,Y,{\bf 1},\omega)$ be a vertex operator algebra \cite{B},
\cite{FLM}. We review various notions of $V$-modules and the definition of rational vertex operator algebras and some basic facts (cf.
\cite{FLM}, \cite{Z}, \cite{DLM3}, \cite{DLM4}).  We also recall intertwining
operator,  fusion rules and some consequences following \cite{FHL}, \cite{ADL}, \cite{W}, \cite{DMZ}, \cite{A}, \cite{DJL}.

\begin{defn} A weak $V$ module is a vector space $M$ equipped
with a linear map
$$
\begin{array}{ll}
Y_M: & V \rightarrow {\rm End}(M)[[z,z^{-1}]]\\
 & v \mapsto Y_M(v,z)=\sum_{n \in \Z}v_n z^{-n-1},\ \ v_n \in {\rm End}(M)
\end{array}
$$
satisfying the following:

1) $v_nw=0$ for $n>>0$ where $v \in V$ and $w \in M$

2) $Y_M( {\textbf 1},z)=Id_M$

3) The Jacobi identity holds:
\begin{eqnarray}
& &z_0^{-1}\de \left({z_1 - z_2 \over
z_0}\right)Y_M(u,z_1)Y_M(v,z_2)-
z_0^{-1} \de \left({z_2- z_1 \over -z_0}\right)Y_M(v,z_2)Y_M(u,z_1) \nonumber \\
& &\ \ \ \ \ \ \ \ \ \ =z_2^{-1} \de \left({z_1- z_0 \over
z_2}\right)Y_M(Y(u,z_0)v,z_2).
\end{eqnarray}
\end{defn}


\begin{defn}
An admissible $V$ module is a weak $V$ module  which carries a
$\Z_+$-grading $M=\bigoplus_{n \in \Z_+} M(n)$, such that if $v \in
V_r$ then $v_m M(n) \subseteq M(n+r-m-1).$
\end{defn}

\begin{defn}
An ordinary $V$ module is a weak $V$ module which carries a
$\C$-grading $M=\bigoplus_{\l \in \C} M_{\l}$, such that:

1) $dim(M_{\l})< \infty,$

2) $M_{\l+n}=0$ for fixed $\l$ and $n<<0,$

3) $L(0)w=\l w=\wt(w) w$ for $w \in M_{\l}$ where $L(0)$ is the
component operator of $Y_M(\omega,z)=\sum_{n\in\Z}L(n)z^{-n-2}.$
\end{defn}

\begin{remark} \ It is easy to see that an ordinary $V$-module is an admissible one. If $W$  is an
ordinary $V$-module, we simply call $W$ a $V$-module.
\end{remark}

We call a vertex operator algebra rational if the admissible module
category is semisimple. We have the following result from
\cite{DLM3} (also see \cite{Z}).

\begin{theorem}\label{tt2.1}
If $V$ is a  rational vertex operator algebra, then $V$ has finitely
many irreducible admissible modules up to isomorphism and every
irreducible admissible $V$-module is ordinary.
\end{theorem}

Suppose that $V$ is a rational vertex operator algebra and let
$M^1,...,M^k$ be the irreducible  modules such that
$$M^i=\oplus_{n\geq 0}M^i_{\l_i+n}$$
where $\l_i\in\Q$ \cite{DLM4},  $M^i_{\l_i}\ne 0$ and each
$M^i_{\l_i+n}$ is finite dimensional.

A vertex operator algebra is called $C_2$-cofinite if $C_2(V)$ has
finite codimension where $C_2(V)=\<u_{-2}v|u,v\in V\>$ (\cite{Z}).

\begin{remark} If $V$ is a vertex operator algebra satisfying $C_{2}$-cofinite
property, $V$ has only finitely many irreducible admissible modules
up to isomorphism \cite{DLM3}, \cite{L}, \cite{Z}.
\end{remark}

\vskip 0.3cm
We now recall the  notion of intertwining operators and fusion
rules from [FHL].
\begin{defn}
Let $V$ be vertex operator algebra and $M^1$, $M^2$, $M^3$ be weak $V$-modules. An intertwining
operator $\mathcal {Y}( \cdot , z)$ of type $\left(\begin{tabular}{c}
$M^3$\\
$M^1$ $M^2$\\
\end{tabular}\right)$ is a linear map$$\mathcal
{Y}(\cdot, z): M^1\rightarrow Hom(M^2, M^3)\{z\}$$ $$v^1\mapsto
\mathcal {Y}(v^1, z) = \sum_{n\in \mathbb{C}}{v_n^1z^{-n-1}}$$
satisfying the following conditions:

(1) For any $v^1\in M^1, v^2\in M^2$and $\lambda \in \mathbb{C},
v_{n+\lambda}^1v^2 = 0$ for $n\in \mathbb{Z}$ sufficiently large.

(2) For any $a \in V, v^1\in M^1$,
$$z_0^{-1}\delta(\frac{z_1-z_2}{z_0})Y_{M^3}(a, z_1)\mathcal
{Y}(v^1, z_2)-z_0^{-1}\delta(\frac{z_1-z_2}{-z_0})\mathcal{Y}(v^1,
z_2)Y_{M^2}(a, z_1)$$
$$=z_2^{-1}\delta(\frac{z_1-z_0}{z_2})\mathcal{Y}(Y_{M^1}(a, z_0)v^1, z_2).$$

(3) For $v^1\in M^1$, $\dfrac{d}{dz}\mathcal{Y}(v^1,
z)=\mathcal{Y}(L(-1)v^1, z)$.
\end{defn}
All of the intertwining operators of type $\left(\begin{tabular}{c}
$M^3$\\
$M^1$ $M^2$\\
\end{tabular}\right)$ form a vector space denoted by $I_V\left(\begin{tabular}{c}
$M^3$\\
$M^1$ $M^2$\\
\end{tabular}\right)$. The dimension of $I_V\left(\begin{tabular}{c}
$M^3$\\
$M^1$ $M^2$\\
\end{tabular}\right)$ is called the
fusion rule of type $\left(\begin{tabular}{c}
$M^3$\\
$M^1$ $M^2$\\
\end{tabular}\right)$ for $V$.

We now have the following result which was essentially proved in
[ADL].

\begin{theorem}\label{n2.2c}
Let $V^1, V^2$ be rational vertex operator algebras. Let $M^1 , M^2,
M^3$ be $V^1$-modules and  $N^1, N^2, N^3$ be $V^2$-modules such
that
$$dim I_{V^1}\left(\begin{tabular}{c}
$M^3$\\
$M^1$ $M^2$\\
\end{tabular}\right)< \infty , \ dim I_{V^2}\left(\begin{tabular}{c}
$N^3$\\
$N^1$ $N^2$\\
\end{tabular}\right)< \infty.$$
Then the linear map
$$\sigma: I_{V^1}\left(\begin{tabular}{c}
$M^3$\\
$M^1$ $M^2$\\
\end{tabular}\right)\otimes I_{V^2}\left(\begin{tabular}{c}
$N^3$\\
$N^1$ $N^2$\\
\end{tabular}\right)\rightarrow I_{V^1\otimes V^2}\left( \begin{tabular}{c}
$M^3\otimes N^3$\\
$M^1\otimes N^1$ $M^2\otimes N^2$\\
\end{tabular}\right)$$
$$\mathcal{Y}_1( \cdot, z)\otimes \mathcal{Y}_2(\cdot, z)\mapsto (\mathcal{Y}_1\otimes
\mathcal{Y}_2)(\cdot, z)$$\\is an isomorphism, where
$(\mathcal{Y}_1\otimes \mathcal{Y}_2)(\cdot, z)$ is defined by
$$(\mathcal{Y}_1\otimes \mathcal{Y}_2)(\cdot, z)(u^1\otimes v^1,z)(u^2\otimes v^2)= \mathcal{Y}_1(u^1,z)u^2\otimes \mathcal{Y}_2(v^1,z)v^2.$$
\end{theorem}

For $m,r,s\in\Z_{+}$ such that $1\leq r\leq m+1$ and $1\leq s\leq m+2$,  let
$$c_{m}=1-\frac{6}{(m+2)(m+3)}$$ and $$h^{(m)}_{r,s}=\frac{[r(m+3)-s(m+2)]^2-1}{4(m+2)(m+3)}.$$

Recall that $L(c_{m},0)$, $m=1,2,\cdots$ are rational (see \cite{W}, \cite{DMZ} and \cite{FF}) and $L(c_{m}, h^{(m)}_{r,s})$ $1\leq r\leq m+1, 1\leq s\leq m+2$ are exactly all the irreducible modules of $L(c_{m},0)$. The fusion rules for these modules are given in \cite{W} (see also \cite{DMZ}, \cite{FF}).

The following lemma can be easily checked  (see \cite{LS}).
\begin{lem}\label{lam-s1}
For $m\in\Z_{+}$, we have

(1) \ $h^{(m)}_{r,s}=h^{(m)}_{m+2-r,m+3-s}$;

(2) \ For any $(p,q)\neq (r,s)$ and $(p,q)\neq (m+2-r,m+3-s)$, we have
$$
h^{(m)}_{r,s}\neq h^{(m)}_{p,q};
$$

(3) \ For all $(r,s)\neq (1,m+2)$ and $(r,s)\neq (m+1,1)$,
$$
h^{(m)}_{r,s}<h^{(m)}_{1,m+2}.$$
\end{lem}

Let $m\in\Z_{+}$. An ordered triple of pairs of integers $((r,s), (r',s'), (r'',s''))$ is called {\em admissible}, if $1\leq r,r',r''\leq m+1$, $1\leq s,s',s''\leq m+2$, $ r+r'+r''\leq 2m+3$, $s+s'+s''\leq 2m+5$, $r<r'+r'', r'<r+r'', r''<r+r', s<s'+s'', s'<s+s'', s''<s+s'$, and both $r+r'+r''$ and $s+s'+s''$ are odd. The following result comes from \cite{W} ( for $c_1=\frac{1}{2}$, also see \cite{DMZ}).
\begin{theorem}\label{wang}
Let $m\in\Z_{+}$. Denote
$$N_{(r,s),(r',s')}^{(r'',s'')}=dim I_{L(c_{m},0)}\left(\begin{tabular}{c}
$L(c_{m},h^{(m)}_{r'',s''})$\\
$L(c_m, h^{(m)}_{r,s})$ $L(c_m, h^{(m)}_{r',s'})$\\
\end{tabular}\right).$$
Then $N_{(r,s),(r',s')}^{(r'',s'')}=1$  if and only if $((r,s), (r',s'), (r'',s''))$ is admissible; otherwise, $N_{(r,s),(r',s')}^{(r'',s'')}$ $=0$. In particular,
$$
dim I_{L(\frac{1}{2},0)}\left(\begin{tabular}{c}
$L(\frac{1}{2},0)$\\
$L(\frac{1}{2},\frac{1}{2})$ $L(\frac{1}{2},\frac{1}{2})$\\
\end{tabular}\right)=1, \  \  I_{L(\frac{1}{2},0)}\left(\begin{tabular}{c}
$L(\frac{1}{2},a)$\\
$L(\frac{1}{2},\frac{1}{2})$ $L(\frac{1}{2},\frac{1}{2})$\\
\end{tabular}\right)=0, \ a=\frac{1}{2}, \ \frac{1}{16},$$
$$
dim I_{L(\frac{1}{2},0)}\left(\begin{tabular}{c}
$L(\frac{1}{2},\frac{1}{16})$\\
$L(\frac{1}{2},\frac{1}{2})$ $L(\frac{1}{2},\frac{1}{16})$\\
\end{tabular}\right)=1, \  \  I_{L(\frac{1}{2},0)}\left(\begin{tabular}{c}
$L(\frac{1}{2},a)$\\
$L(\frac{1}{2},\frac{1}{2})$ $L(\frac{1}{2},\frac{1}{16})$\\
\end{tabular}\right)=0, \ a=0, \ \frac{1}{2}.$$

\end{theorem}

\begin{cor}\label{wang-lam-co}
For $n\geq 1$, $0\leq 2k,2l,2p\leq n$, $0\leq 2l_1,2l_2\leq n+1$ such that $l_1\neq l_2$,
$$I_{L(c_{n},0)}\left(\begin{tabular}{c}
$L(c_{n},h^{(n)}_{2p+1,2l_1+1})$\\
$L(c_n, h^{(n)}_{2k+1,1})$ $L(c_n, h^{(n)}_{2l+1,2l_2+1})$\\
\end{tabular}\right)=0.$$

\end{cor}
\pf Since $l_1\neq l_2$, by the definition of admissible triple of pairs and (1) of Lemma \ref{lam-s1}, $((2k+1,1),(2l+1,2l_2+1),(2p+1,2l_1+1))$ can not be  an admissible triple. Then the corollary follows from Theorem \ref{wang}.\qed

\vskip 0.3cm

Now let $V$ be a vertex operator algebra and $M^1, M^2$ be weak $V$-modules, we call a weak $V$-module
$M$ an extension of $M^2$ by $M^1$ if there is a short exact
sequence
$$0\rightarrow M^1\rightarrow M \rightarrow M^2\rightarrow 0.$$
Then we could define the equivalence of two extensions, and then
define the extension group $Ext_V^1(M^2, M^1)$. Recall that $V$ is
called $C_2$-cofinite if the subspace $C_2(V)$ of $V$ spanned by
$u_{-2}v$ for $u,v\in V$ has finite codimension. We have the following result from \cite{A}.

\begin{theorem}\label{n2.2}
Let $V$ be a $C_2$-cofinite vertex operator algebra, then $V$ is
rational if and only if $Ext_V^1(N, M) = 0$ for any pair of
irreducible $V$-modules $M$ and $N$.
\end{theorem}

We will also need the following lemma which was proved in \cite{DJL}.
\begin{lem}\label{la1}
Let $V$ be a vertex operator algebra and $U$ a rational vertex operator subalgebra
of $V$ with the same Virasoro element. Let $M^1,M^2$ be irreducible $V$-modules.
Assume that
$$I_{U}\left(\begin{tabular}{c}
$N^1$\\
$N$\  $N^2$
\end{tabular}\right)=0$$
for any irreducible $U$-submodules $N^1,N,N^2$ of $M^1,V,M^2,$ respectively.
Then $$Ext_V^1(M^2, M^1) = 0.$$
\end{lem}

\section{The Zhu algebra   of $M^{(n)}$}
\def\theequation{3.\arabic{equation}}
\setcounter{equation}{0}

From this section, we always assume that ${\frak g}=\frak{sl}_{2}(\C)$.  Let $\widehat{\frak g}$ be the corresponding affine Lie algebra. For $l\in\C$, let $L_{\widehat{\frak g}}(l,0)$ be the simple vertex operator algebra associated to the level $l$ irreducible highest weight representation of $\widehat{\frak g}$. Then $L_{\widehat{\frak g}}(1,0)\cong V_{A_{1}}$, where $V_{A_{1}}$ is the lattice vertex operator algebra associated to the lattice $A_{1}=\Z\al$ with $(\al,\al)=2$.

Let $A_{1}^{n+1}=\Z\alpha^{0}\oplus\Z\al^{1}\oplus\cdots \oplus\Z\al^n$ be the orthogonal sum of $n+1$ copies of $A_{1}$ and $V_{A_{1}^{n+1}}$ the lattice vertex operator algebra associated with the lattice $A_{1}^{n+1}$. Then as a vertex operator algebra, we have
\begin{eqnarray*}
&  V_{A_{1}^{n+1}}\cong V_{A_{1}}\otimes V_{A_{1}}\otimes\cdots\cdots\otimes V_{A_1}\\
& \cong L_{\wg}(1,0)\otimes L_{\wg}(1,0)\otimes\cdots\otimes L_{\wg}(1,0)\\
& =\oplus_{0\leq 2l\leq n+1}L_{\wg}(n+1,2l)\otimes M^{0^{n+1}}(2l).
\end{eqnarray*}
It is known that the commutator $M^{{0^{n+1}}}(0)$ of $L_{\wg}(n+1,0)$ in $V_{A_{1}^{n+1}}$ is a simple vertex operator algebra and $M^{0^{n+1}}(2l)$, $0\leq 2l\leq n+1$ are irreducible modules of $M^{0^{n+1}}(0)$. Furthermore, for $a=(a_{0},a_{1},\cdots,a_{n})\in\Z_{2}^{n+1}$ and $\gamma_{a}=\frac{1}{2}\sum\limits_{i=0}^{n}a_{i}\al^{i}$, we have
\begin{eqnarray*}
& V_{\gamma_{a}+A_{1}^{n+1}}\cong L_{\wg}(1,a_{0})\otimes L_{\wg}(1,a_{1})\otimes\cdots\otimes L_{\wg}(1,a_{n})\\
& = \oplus_{0\leq l\leq n+1, l-|\gamma_{a}|\in 2\Z}L_{\wg}(n+1,l)\otimes M^{{\gamma_{a}}}(l),
\end{eqnarray*} and  $M^{{\gamma_{a}}}(l)$ are also irreducible modules of $M^{0^{n+1}}(0)$, where $|\gamma_{a}|=\sum\limits_{i=0}^{n}a_{i}$.   For short denote
$$
M^{(n)}=M^{0^{n+1}}(0).
$$
Let
$$Q=span_{\Z}\{\al^0-\al^1,\al^1-\al^2,\cdots,\al^n-\al^{n+1}\},$$
and
$$
\Phi=\{\frac{\pm(\al^i-\al^j)}{\sqrt{2}}| 0\leq i<j\leq n\}.$$
Then $Q$ is isomorphic to the lattice $\sqrt{2}A_{n}$ and $\Phi$ is a root system  of type $A_{n}$. For $0\leq i,j\leq n$, $i\neq  j$,  let
$$
\omega^{ij}=\frac{1}{16}(\al^i-\al^j)(-1)(\al^i-\al^j)(-1){\bf 1}-\frac{1}{4}(e^{\al^i-\al^j}+e^{-\al^i+\al^j}).
$$

It is obvious that $\omega^{ij}=\omega^{ji}$. The following lemma was obtained in \cite{LS}.
\begin{lem}\label{l3.1} (1) For $0\leq i<j\leq n$, $\omega^{ij}$ are conformal vectors of central charge $\frac{1}{2}$ in $M^{(n)}$.

(2) The vertex operator algebra $M^{(n)}$ is generated by $\omega^{ij}$, $0\leq i<j\leq n$.

(3) $\omega=\frac{4}{n+3}\sum\limits_{0\leq i<j\leq n}\omega^{ij}$ is the Virasoro vector of $M^{(n)}$.

(4) $M^{(n)}$ is $C_{2}$-cofinite and simple.
\end{lem}

As in \cite{LS}, for any $\delta=(\delta_{0},\cdots,\delta_{n-1})\in\Z_{2}^{n}$ and $0\leq k\leq n+1$, define
$$
\delta'=\left\{
\begin{array}{l}
 (\delta_{0},\cdots,\delta_{n-1},0), \ \ {\rm if} \ |\delta|-k\in 2\Z\\
(\delta_{0},\cdots,\delta_{n-1},1), \ \ {\rm if} \ |\delta|-k\in 2\Z+1
\end{array}
\right.
$$
and define
\begin{equation}\label{de}
M^{\delta}(k)=M^{\delta'}(k),
\end{equation}
where $M^{\delta'}(k)$ is defined as above.
Then we have the following lemma from \cite{LS}
\begin{lem}\label{lam-s2}
Let $\delta,\sigma\in\Z_{2}^{n}$, $1^n=(1,1,\cdots,1)\in\Z_{2}^n$ and $0\leq k,l\leq n+1$. Then $M^{\delta}(k)\cong M^{\sigma}(l)$ if and only if either $\delta=\sigma$ and $k=l$ or $l=n+1-k$ and $\sigma=\delta+1^{n}$.
\end{lem}

For  $0\leq i<j<l\leq n$, let $M^{(ijl)}$ be the vertex operator subalgebra of $M^{(n)}$  generated by $\omega^{ij},\omega^{jl},\omega^{il}$. Then  $M^{(ijl)}$ is isomorphic to the vertex operator algebra
$$
L(\frac{7}{10},0)\otimes L(\frac{1}{2},0)\oplus L(\frac{7}{10}, \frac{3}{2})\otimes L(\frac{1}{2},\frac{1}{2})
$$
studied in \cite{CL}. It was pointed out in \cite{CL} that $M^{(ijl)}$ is rational and has eight inequivalent irreducible modules including one whose lowest weight space is two-dimensional. The following lemma is easy to check (see also \cite{CL}).
 \begin{lem}\label{l4.1}
 Denote by  $W^i, 1\leq i\leq 8$ the lowest weight spaces of the  inequivalent irreducible modules of $M^{(ijl)}$  such that $W^{1}$ is two-dimensional and $W^i, 2\leq i\leq 8$ are one-dimensional.  Then there is a basis of $W^{(1)}$ such that the matrices of $\omega^{ij},\omega^{il}, \omega^{jl}$ are
$$
\left[\begin{array}{ccc} 0 & 0\\
0 & \frac{1}{2}
\end{array}\right]£¬ \ \  \ \left[\begin{array}{cc} \frac{3}{8} & \frac{1}{8}\\
\frac{3}{8} & \frac{1}{8}
\end{array}\right]£¬ \ \ \ \left[\begin{array}{cc} \frac{3}{8} & -\frac{1}{8}\\
-\frac{3}{8} &  \frac{1}{8}
\end{array}\right]
$$
respectively and the action of $(\omega^{ij},\omega^{il}, \omega^{jl})$ on $W^{i}, 2\leq i\leq 8$ are
$$
(0,0,0), (0,\frac{1}{16},\frac{1}{16}), \ (\frac{1}{16},0,\frac{1}{16}), \ (\frac{1}{16},\frac{1}{16},0), \ (\frac{1}{2},\frac{1}{16},\frac{1}{16}), \ (\frac{1}{16},\frac{1}{2},\frac{1}{16}), \ (\frac{1}{16},\frac{1}{16},\frac{1}{2})
$$
respectively. In particular,  $\omega^{pq}(\omega^{pq}-\frac{1}{2})=0$ on $W^1$, where $pq=ij,jl$ or $il$.
 \end{lem}

  Recall that \cite{LS}, \cite{LY}
  $$
  \omega=\frac{4}{n+3}\sum\limits_{0\leq i< j\leq n}\omega^{ij}.
  $$
   For distinct $0\leq i,j,l\leq n$, denote $$\omega^{ijl}=\frac{4}{5}(\omega^{ij}+\omega^{jl}+\omega^{il}).$$
  The following lemma can be checked directly.
\begin{lem}\label{l4.9} For distinct $0\leq i,j,l,k\leq n$, we have
\begin{equation}\label{l4-e1}
(\omega, \omega^{ij})=(\omega^{ijl},\omega^{ij})=(\omega^{ij},\omega^{ij})=\frac{1}{4};
\end{equation}
\begin{equation}\label{l4-e2}
(\omega^{ijl},\omega^{ijl})=\frac{3}{5}, \ (\omega^{jl},\omega^{ij})=\frac{1}{32};
\end{equation}
\begin{equation}\label{l7-e1}
\omega^{ij}_1\omega^{jl}=\frac{1}{4}[\omega^{ij}+\omega^{jl}-\omega^{il}];
\end{equation}
\begin{equation}\label{ee1}
\omega^{ij}_{1}(\omega^{il}-\omega^{jl})=\frac{1}{2}(\omega^{il}-\omega^{jl}), \  \ \omega^{ij}_{p}(\omega^{il}-\omega^{jl})=0, \ \ p\geq 2;
\end{equation}
\begin{equation}\label{ee2}
\omega^{ij}_{p}(\omega^{ijl}-\omega^{ij})=0, \   \ p\geq 0;
\end{equation}
\begin{equation}\label{l4-e8}
\omega^{ij}_{0}\omega^{ij}=L(-1)\omega^{ij}=\omega^{ijl}_{0}\omega^{ij};
\end{equation}
\begin{equation}\label{l4-e9}
(\omega^{ij}_{0}(\omega^{il}-\omega^{jl}))_{0}\omega^{ij}=[-\frac{4}{3}\omega^{ij}_{-1}+\frac{3}{2}L(-1)\omega^{ij}_{0}-\frac{1}{2}L(-1)^{2}](\omega^{il}-\omega^{jl}).
\end{equation}

\end{lem}

\begin{lem}\label{l4.10}
$M^{(n)}$ is linearly spanned by
\begin{equation}\label{l4-e7}
\omega^{i_{1}j_{1}}_{n_{1}}\omega^{i_{2}j_{2}}_{n_{2}}\cdots \omega^{i_{s}j_{s}}_{n_{s}}{\bf 1},
\end{equation}
where $s\geq 0$, $n_{k}\leq 0$ and $0\leq i_{k}<j_{k}\leq n$, $k=1,2,\cdots,s$.
\end{lem}
\pf   We denote by $M'^{(n)}$ the subspace of $M^{(n)}$ linearly spanned by (\ref{l4-e7}). It suffices to prove that $M^{(n)}\subseteq M'^{(n)}$. By Proposition 4.6 in \cite{LS}, $M^{(n)}$ is generated by $\omega^{ij}$, where $0\leq i<j\leq n$. So it is enough to prove that  for any homogeneous $u\in M^{(n)}$ such that
$$
u=\omega^{i_{1}j_{1}}_{n_{1}}\omega^{i_{2}j_{2}}_{n_{2}}\cdots \omega^{i_{m}j_{m}}_{n_{m}}{\bf 1},
$$
we have $u\in M'^{(n)}$, where $m\geq 1$, $n_{1}\geq 1,n_{l}\leq 0$, $2\leq l\leq m$ and $0\leq i_{k}<j_{k}\leq n$, $k=1,2,\cdots,m$. Denote
 $$
 y^1=\omega^{i_{2}j_{2}}_{n_{2}}\cdots \omega^{i_{m}j_{m}}_{n_{m}}{\bf 1}.
 $$
 We prove this by induction  on $wt(y^1)$. Obviously, if $wt(y^1)=2$, then $u\in M'^{(n)}$. Assume for all such homogeneous $u$ satisfying $wt(y^1)< N$, $u\in M'^{(n)}$. Now suppose that $wt(y^1)=N$.
By Jacobi identity,
$$
\begin{array}{ll}
& \omega^{i_{1}j_{1}}_{n_{1}}\omega^{i_{2}j_{2}}_{n_{2}}\cdots \omega^{i_{m}j_{m}}_{n_{m}}{\bf 1}\\
= & \omega^{i_{2}j_{2}}_{n_{2}}\omega^{i_{1}j_{1}}_{n_{1}}\omega^{i_{3}j_{3}}_{n_{3}}\cdots \omega^{i_{m}j_{m}}_{n_{m}}{\bf 1}+ \sum\limits_{q=0}^{n_{1}}
\left(\begin{array}{c}n_{1}\\q\end{array}\right)(\omega^{i_{1}j_{1}}_{q}\omega^{i_{2}j_{2}})_{n_{1}+n_{2}-q}\omega^{i_{3}j_{3}}_{n_{3}}\cdots \omega^{i_{m}j_{m}}_{n_{m}}{\bf 1}.
\end{array}
$$
 Note that $n_{1}\geq 1$ and for $q\geq 1$, $\omega^{i_{1}j_{1}}_{q}\omega^{i_{2}j_{2}}\in V_{2}+V_{0}$ and $wt(\omega^{i_{3}j_{3}}_{n_{3}}\cdots \omega^{i_{m}j_{m}}_{n_{m}}{\bf 1})\leq wt(y^1)$, by inductive assumption,
$$
\omega^{i_{1}j_{1}}_{n_{1}}\omega^{i_{3}j_{3}}_{n_{3}}\cdots \omega^{i_{m}j_{m}}_{n_{m}}{\bf 1}\in M'^{(n)}, \ (\omega^{i_{1}j_{1}}_{q}\omega^{i_{2}j_{2}})_{n_{1}+n_{2}-q}\omega^{i_{3}j_{3}}_{n_{3}}\cdots \omega^{i_{m}j_{m}}_{n_{m}}{\bf 1}\in M'^{(n)}, \ q\geq 1.
$$
So we only need to prove that
 \begin{equation}\label{add-1}
 (\omega^{i_{1}j_{1}}_{0}\omega^{i_{2}j_{2}})_{n_{1}+n_{2}}\omega^{i_{3}j_{3}}_{n_{3}}\cdots \omega^{i_{m}j_{m}}_{n_{m}}{\bf 1}\in M'^{(n)}.
 \end{equation}

 If $n_2<0$, then
 $$
 \begin{array}{ll}
 &(\omega^{i_{1}j_{1}}_{0}\omega^{i_{2}j_{2}})_{n_{1}+n_{2}}\omega^{i_{3}j_{3}}_{n_{3}}\cdots \omega^{i_{m}j_{m}}_{n_{m}}{\bf 1}\\
 =& \omega^{i_{1}j_{1}}_{0}\omega^{i_{2}j_{2}}_{n_{1}+n_{2}}\omega^{i_{3}j_{3}}_{n_{3}}\cdots \omega^{i_{m}j_{m}}_{n_{m}}{\bf 1} -\omega^{i_{2}j_{2}}_{n_{1}+n_{2}}\omega^{i_1j_1}_{0}\omega^{i_{3}j_{3}}_{n_{3}}\cdots \omega^{i_{m}j_{m}}_{n_{m}}{\bf 1}
 \end{array}
 $$
 Then $wt(\omega^{i_1j_1}_{0}\omega^{i_{3}j_{3}}_{n_{3}}\cdots \omega^{i_{m}j_{m}}_{n_{m}}{\bf 1})<wt(y^1)=N$. Obviously, $(\omega^{i_{3}j_{3}}_{n_{3}}\cdots \omega^{i_{m}j_{m}}_{n_{m}}{\bf 1})<N$, so by inductive assumption,
 $$
 \omega^{i_{2}j_{2}}_{n_{1}+n_{2}}\omega^{i_1j_1}_{0}\omega^{i_{3}j_{3}}_{n_{3}}\cdots \omega^{i_{m}j_{m}}_{n_{m}}{\bf 1}\in M'^{(n)}, \
 \omega^{i_1j_1}_{0}\omega^{i_{2}j_{2}}_{n_{1}+n_{2}}\omega^{i_{3}j_{3}}_{n_{3}}\cdots \omega^{i_{m}j_{m}}_{n_{m}}{\bf 1}\in M'^{(n)}.
 $$
 Therefore we may assume that $n_2=0$. That is, it suffices to prove that
\begin{equation}\label{addi-1}
 (\omega^{i_{1}j_{1}}_{0}\omega^{i_{2}j_{2}})_{n_{1}}\omega^{i_{3}j_{3}}_{n_{3}}\cdots \omega^{i_{m}j_{m}}_{n_{m}}{\bf 1}\in M'^{(n)}.
 \end{equation}

  For short, we denote $i=i_1$, $j=j_1$. Recall  that for $0\leq l\leq n$ such that $l\neq i,j$, $\omega^{ijl}=\frac{4}{5}(\omega^{ij}+\omega^{jl}+\omega^{il})$. By the fact that the weight two subspace  $M^{(n)}_{2}$ is linearly spanned by $\omega^{pq}$, $0\leq p,q\leq n$, $p\neq q$,
 it is easy to see that for fixed $0\leq i\neq j\leq n$, $M^{(n)}_{2}$ is linearly spanned by
 $$\Omega^{ij}=\{\omega^{ij},\ \omega^{ijl}-\omega^{ij}, \ \omega^{il}-\omega^{jl}, \ \omega^{kl},  | \  0\leq k\neq l\leq n, \ k,l\neq i,j\}.$$
Let $L^{(ij)}(\frac{1}{2},0)$ be the Virasoro vertex operator algebra generated by $\omega^{ij}$. For $k,l\neq i,j$, by Lemma \ref{l4.9}, $\omega^{ijl}-\omega^{ij},\ \omega^{kl}$ generate irreducible modules of $L^{(ij)}(\frac{1}{2},0)$ isomorphic to itself respectively,  and $\omega^{il}-\omega^{jl}$ generates an irreducible $L^{(ij)}(\frac{1}{2},0)$-module isomorphic to $L(\frac{1}{2},\frac{1}{2})$.   Note that for $k,l\neq i,j$,
$$
\omega^{ij}_0(\omega^{ijl}-\omega^{ij})=\omega^{ij}_0(\omega^{kl})=0,
$$
and for any $m\in\Z$,
$$
(\omega^{ij}_{0}\omega^{ij})_m=-m\omega^{ij}_{m-1}.
$$
Then it is enough to prove that
 $$
 (\omega^{ij}_{0}(\omega^{il}-\omega^{jl}))_{n_{1}}x^3_{n_3}\cdots x^m_{n_m}{\bf 1}\in M'^{(n)},
 $$
 where $n_1,\cdots, n_m$ are the same as above, and
 $x^3,\cdots,x^m\in\Omega^{ij}.$

 Denote
 $$
 y^2=x^{3}_{n_{3}}\cdots x^{m}_{n_{m}}{\bf 1}.
 $$
 Obviously, $wt(y^2)<N$.
Let $W$ be the $L(\frac{1}{2},0)$-submodule of $M^{(n)}$ generated by $y^2$. Then
$$
W^{ij}=span\{\omega^{ij}_{k_1}\cdots\omega^{ij}_{k_t}x^{3}_{n_{3}}\cdots x^{m}_{n_{m}}{\bf 1}| t\geq0, \ k_p\in\Z, \ 1\leq p\leq t, k_1\leq k_2\leq\cdots\leq k_t\}.$$
Note that
 $W^{ij}$ is a direct sum of irreducible $L^{(ij)}(\frac{1}{2},0)$-modules. Then we may assume that
$$
y^2=\sum\limits_{q=1}^{t}v^q,
$$
 for some $t\geq 1$,  where for $1\leq q\leq t$, $v^q$ is a non-zero element in an irreducible $L^{(ij)}(\frac{1}{2},0)$-module generated by a lowest weight vector $v'^{q}$. By inductive assumption,
$$
W^{ij}\subseteq M'^{(n)}.
$$
Therefore $v'^q, v^q\in M'^{(n)}$ and  $wt(v'^q)<N$, $wt(v^q)<N$,  $q=1,2,\cdots,t$. So by inductive assumption, for $k\in\Z$,
\begin{equation}\label{claim1}(\omega^{il}-\omega^{jl})_{k}v^q\in M'^{(n)},\  \ (\omega^{il}-\omega^{jl})_{k}v'^q\in M'^{(n)}.\end{equation}

  For $1\leq q\leq t$,
if $v^q\notin \C v'^q$, then $v^q$ is a linear combination of elements of the form:
$$
\omega^{ij}_{-a_{1}}\cdots \omega^{ij}_{-a_{k}}v'^q,
$$
where $k\geq 1$, $a_{1},\cdots, a_{k}\geq 0$. Since
$$
\begin{array}{ll}
& (\omega^{ij}_{0}(\omega^{il}-\omega^{jl}))_{n_1}\omega^{ij}_{-a_{1}}\cdots \omega^{ij}_{-a_{k}}v'^q\\
=& \omega^{ij}_{-a_{1}}(\omega^{ij}_0(\omega^{il}-\omega^{jl}))_{n_1}\omega^{ij}_{-a_{2}}\cdots \omega^{ij}_{-a_{k}}v'^q
+((\omega^{ij}_{0}(\omega^{il}-\omega^{jl}))_0\omega^{ij})_{n_1-a_{1}}\omega^{ij}_{-a_{2}}\cdots \omega^{ij}_{-a_{k}}v'^q\\
& +\sum\limits_{r=1}^{n_1}\left(\begin{array}{c}n_{1}\\r\end{array}\right)((\omega^{ij}_{0}(\omega^{il}-\omega^{jl}))_r\omega^{ij})_{n_1-a_{1}-r}\omega^{ij}_{-a_{2}}\cdots \omega^{ij}_{-a_{k}}v'^q,
\end{array}
$$
by inductive assumption and (\ref{l4-e9}), we see that $((\omega^{ij}_0(\omega^{il}-\omega^{jl}))_{n_1}v^q\in M'^{(n)}$.

If $v^q\in\C v'^q$, we may assume that $v^q=v'^q$. Then $v^q$ is a lowest weight vector of the Virasoro vertex operator algebra $L^{(ij)}(\frac{1}{2},0)$ such that
$$
\omega^{ij}_1v^q=a v^q,$$
where $a=0, \frac{1}{2}$ or $\frac{1}{16}$ (Actually, $a=0$ or $\frac{1}{2}$ by using fusion rules of the vertex operator algebra  $L^{(ij)}(\frac{1}{2},0)$).
 For $r\in\Z$, $r\leq wt(v^q)+1$, set
$$
\begin{array}{ll}
U^{r}=&span\{\omega^{ij}_{-k_1}\cdots\omega^{ij}_{-k_b}(\omega^{il}-\omega^{jl})_{k}v^q, \\
&\omega^{ij}_{-k_1}\cdots\omega^{ij}_{-k_b}(\omega^{ij}_0(\omega^{il}-\omega^{jl}))_{k+1}v^q| b\geq 0, k_1\geq\cdots\geq k_b\geq 0, k\geq r\}.
\end{array}
$$
Since $\omega^{ij}_0\omega^{ij}_0(\omega^{il}-\omega^{jl})=\frac{4}{3}\omega^{ij}_{-1}(\omega^{il}-\omega^{jl})$, it is easy to check that $U^r$ is an $L^{(ij)}(\frac{1}{2},0)$-submodule of $M^{(n)}$. It is obvious that $U^{r+1}\subseteq U^{r}$, for $r\in\Z$, $r\leq wt(v^q)+1$.
 We have the following claim.

{\bf Claim} \ For $r\leq wt(v^q)+1$,
$$
U^{r}=span\{\omega^{ij}_{-k_1}\cdots\omega^{ij}_{-k_b}(\omega^{il}-\omega^{jl})_{k}v^q, | b\geq 0, k_1\geq\cdots\geq k_b\geq 0, k\geq r\}.
$$

{\em Proof of the claim} \ Let $h\leq wt(v^q)+1$ be such that
$$
(\omega^{il}-\omega^{jl})_{h}v^q\neq 0, \ (\omega^{il}-\omega^{jl})_{h+k}v^q=0, \ k\geq 1.
$$
Then
$$
\omega^{ij}_1(\omega^{il}-\omega^{jl})_{h+k}v^q=0, \ k\geq 1.
$$
On the other hand, for $k\geq 1$,
$$
\begin{array}{ll}
& \omega^{ij}_1(\omega^{il}-\omega^{jl})_{h+k}v^q\\
=& \sum\limits_{r=0}^{1}\left(\begin{array}{c}1\\r\end{array}\right)((\omega^{ij}_r(\omega^{il}-\omega^{jl}))_{h+k+1-r}v^q+(\omega^{il}-\omega^{jl})_{h+k}\omega^{ij}_1v^q\\ 
=& (\omega^{ij}_0(\omega^{il}-\omega^{jl}))_{h+k+1}v^q+(\frac{1}{2}+a)(\omega^{il}-\omega^{jl})_{h+k}v^q.
\end{array}
$$
So we have
$$
(\omega^{ij}_0(\omega^{il}-\omega^{jl}))_{h+k}v^q=0, \ k\geq 2.
$$
Then for $k\geq 2$, 
$$
\begin{array}{ll}
& \omega^{ij}_k(\omega^{il}-\omega^{jl})_{h}v^q\\
= & \sum\limits_{r=0}^{k}\left(\begin{array}{c}k\\r\end{array}\right)((\omega^{ij}_r(\omega^{il}-\omega^{jl}))_{h+k-r}v^q\\
=& (\omega^{ij}_0(\omega^{il}-\omega^{jl}))_{h+k}v^q+\frac{k}{2}(\omega^{il}-\omega^{jl})_{h+k-1}v^q\\
= & 0.
\end{array}
$$
This means that $(\omega^{il}-\omega^{jl})_hv^q$ is a lowest weight vector of $L^{(ij)}(\frac{1}{2},0)$. So
$$
\omega^{ij}_1(\omega^{il}-\omega^{jl})_{h}v^q=a^{(h)}(\omega^{il}-\omega^{jl})_{h}v^q,
$$
for some $a^{(h)}\in \{0,\frac{1}{2},\frac{1}{16}\}$.
Note that
$$
\omega^{ij}_1(\omega^{il}-\omega^{jl})_{h}v^q=(\omega^{ij}_0(\omega^{il}-\omega^{jl}))_{h+1}v^q+(\frac{1}{2}+a)(\omega^{il}-\omega^{jl})_hv^q.
$$
So
$$
(\omega^{ij}_0(\omega^{il}-\omega^{jl}))_{h+1}v^q=(a^{(h)}-\frac{1}{2}-a)(\omega^{il}-\omega^{jl})_{h}v^q,
$$
where $a$ is the same as above.
We have shown  that the claim holds for $r\geq h$. Now assume that for $r+1$, the claim is true. Then
$$
U^{r+1}=span\{\omega^{ij}_{-k_1}\cdots\omega^{ij}_{-k_b}(\omega^{il}-\omega^{jl})_{k}v^q| b\geq 0, k_1\geq\cdots\geq k_b\geq 0, k\geq r+1\}.
$$
We will prove that the claim is true for $r$. Consider the $L^{(ij)}(\frac{1}{2},0)$-module
$U^r/U^{r+1}$. Since
$$
\omega^{ij}_{2}(\omega^{il}-\omega^{jl})_{r}v^q=(\omega^{ij}_0(\omega^{il}-\omega^{jl}))_{r+2}v^q+(\omega^{il}-\omega^{jl})_{r+1}v^q\in U^{r+1}
$$
and
$$
\omega^{ij}_{3}(\omega^{il}-\omega^{jl})_{r}v^q=(\omega^{ij}_0(\omega^{il}-\omega^{jl}))_{r+3}v^q+\frac{3}{2}(\omega^{il}-\omega^{jl})_{r+2}v^q\in U^{r+1},$$
it follows that the image of $(\omega^{il}-\omega^{jl})_{r}v^q$ in $U^r/U^{r+1}$ is a lowest weight vector of $L^{(ij)}(\frac{1}{2},0)$. So there exists $v\in U^{r+1}$ such that
$$
\omega^{ij}_1(\omega^{il}-\omega^{jl})_{r}v^q=a^{(r)}(\omega^{il}-\omega^{jl})_{r}v^q+v.
$$
where $a^{(r)}=0$ or $\frac{1}{2}$. On the other hand,
$$
\omega^{ij}_1(\omega^{il}-\omega^{jl})_{r}v^q=(\omega^{ij}_0(\omega^{il}-\omega^{jl}))_{r+1}v^q+(\frac{1}{2}+a)(\omega^{il}-\omega^{jl})_{r}v^q.
$$
We deduce that
 $$(\omega^{ij}_0(\omega^{il}-\omega^{jl}))_{r+1}v^q\in span\{\omega^{ij}_{-k_1}\cdots\omega^{ij}_{-k_b}(\omega^{il}-\omega^{jl})_{k}v^q| b\geq 0, k_1\geq\cdots\geq k_b\geq 0, k\geq r\}.$$
 This shows  that the claim holds for $r$, since for $k\geq 2$, $(\omega^{ij}_0(\omega^{il}-\omega^{jl}))_{r+k}v^q\in U^{r+1}$. The proof of the claim is complete.

Now in the claim let $r=n_1-1$, then
$$
(\omega^{ij}_0(\omega^{il}-\omega^{jl}))_{n_1}v^q\in U^{n_1-1}.
$$
By the claim
$$
(\omega^{ij}_0(\omega^{il}-\omega^{jl}))_{n_1}v^q\in span\{\omega^{ij}_{-k_1}\cdots\omega^{ij}_{-k_b}(\omega^{il}-\omega^{jl})_{k}v^q| b\geq 0, k_1\geq\cdots\geq k_b\geq 0, k\geq n_1-1\}.
$$
By (\ref{claim1}), $(\omega^{ij}_0(\omega^{il}-\omega^{jl}))_{n_1}v^q\in M'^{(n)}$. Therefore $u\in M'^{(n)}$.  We complete the proof of the lemma. \qed

\vskip 0.3cm
Let $A(M^{(n)})$ be the Zhu algebra of $M^{(n)}$. We denote the multiplication in $A(M^{(n)})$ by $*$. For $u\in M^{(n)}$, we denote the image of $u$ in $A(M^{(n)})$ by $[u]$. For distinct $0\leq i,j,k,l\leq n$, by Lemma \ref{l4.9}, we have
\begin{equation}\label{zhu-sym}
[\omega^{ij}]*[\omega^{kl}]=[\omega^{kl}]*[\omega^{ij}].
\end{equation}

\begin{theorem}\label{p4.2}
The Zhu algebra $A(M^{(n)})$ is generated by $[\omega^{ij}], 0\leq i,j\leq n$.
\end{theorem}
\pf
Let $A'$ be the subalgebra of $A(M^{(n)})$ generated by $[\omega^{ij}]$, $0\leq i<j\leq n$. It suffices to show that for every homogeneous $u\in M^{(n)}$, $[u]\in A'$. We will approach it by induction on the weight of $u$. If $wt(u)=2$, it is obvious that $[u]\in A'$. Suppose that for all homogeneous $u$ such that $wt(u)\leq m-1$, we have $[u]\in A'$. Now suppose $wt(u)=m$. By Lemma \ref{l4.10}, we may assume that
$$
u=\omega^{i_{1}j_{1}}_{n_{1}}\omega^{i_{2}j_{2}}_{n_{2}}\cdots \omega^{i_{s}j_{s}}_{n_{s}}{\bf 1},
$$
for some $n_{k}\leq 0$ and $0\leq i_{k}<j_{k}\leq n$, $k=1,2,\cdots,s$. Since for $p\geq 0$, $[\omega^{i_{1}j_{1}}_{-p-2}+2\omega^{i_{1}j_{1}}_{-p-1}+\omega^{i_{1}j_{1}}_{-p})
\omega^{i_{2}j_{2}}_{n_{2}}\cdots \omega^{i_{s}j_{s}}_{n_{s}}{\bf 1}]=[0]$, we may assume that $-1\leq n_{1}\leq 0$. Denote $u^2=\omega^{i_{2}j_{2}}_{n_{2}}\cdots \omega^{i_{s}j_{s}}_{n_{s}}{\bf 1}$. Then by inductive assumption, $[u^2], [\omega^{i_{1}j_{1}}_{1}u^2]\in A'$.  Since
$$
[(\omega^{i_{1}j_{1}}_{0}+\omega^{i_{1}j_{1}}_{1})u^2]=[\omega^{i_{1}j_{1}}]*[u^2]-[u^2]*[\omega^{i_{1}j_{1}}],$$
it follows that $[\omega^{i_{1}j_{1}}_{0}u^2]\in A'$. Then by the fact that
$$[\omega^{i_{1}j_{1}}]*[u^2]=[(\omega^{i_{1}j_{1}}_{-1}+2\omega^{i_{1}j_{1}}_{0}+\omega^{i_{1}j_{1}}_{1})u^2],$$
we have $[u]\in A'$. Therefore $A(M^{(n)})=A'$.
\qed

\section{Classification of irreducible modules  of $M^{(n)}$}
\def\theequation{4.\arabic{equation}}
\setcounter{equation}{0}

In this section, we will  classify all irreducible modules of $M^{(n)}$. It turns out all the inequivalent irreducible modules of $M^{(n)}$ are the ones given in \cite{LS}.

For $N\geq 1$, let $S_{N}$ be the $N$-symmetric group generated by the transpositions $s_{i}=(i,i+1)$, $ i=1,2,\cdots,N-1$. Denote $s_{i,j}=(i,j)$.  Let $\C[S_{N}]$ be the group algebra of $S_{N}$. For $N\geq 3$,  let $J^{N}$ be  the two-sided ideal of $\C[S_{N}]$ generated by $(s_{i}+s_{i+1}+s_{i+1}s_{i}s_{i+1})(s_{i}+s_{i+1}+s_{i+1}s_{i}s_{i+1}-3), 1\leq i\leq N-1$. If $N=1,2$, let $J^{N}=0$. Define
$$
T^{N}=\C[S_{N}]/J^{N}.
$$

It is known that $\C[S_{N}]$ is semi-simple and the number of  minimal ideals of $\C[S_{N}]$ equals the number of partitions of $N$ ( see \cite{Ja}). Let $\lambda=(\lambda_{1},\cdots,\lambda_{k})$ be a partition of $N$  with $\lambda_{1}\geq \lambda_{2}\geq\cdots\geq \lambda_{k}$.  We still use
$\lambda$ to denote the associated Young Diagram, which is a frame with $\lambda_{i}$ boxes in the $i$-th row and the rows of boxes lined up on the left. A $\lambda$-tableau $\bft$ is the Young diagram with each box filled with a number from $\{ 1, 2, \cdots, N\}$ without any repetition. Given a tableau $\bf t$, we define two subgroups of $S_{N}$ as follows:
$$
R_{\bft}=\{\sigma\in S_{N}|\ \sigma \ \text{permutes the numbers in each row of $\bft$}\}
$$
and
$$
C_{\bft}=\{\sigma\in S_{N}| \ \sigma \ \text{permutes the  numbers  in each  column of $\bft$}\}.
$$
Introduce two elements corresponding to the two subgroups as follows:
$$
a_{\bft}=\sum\limits_{\sigma\in R_{\bft}}\sigma, \ b_{\lambda}=\sum\limits_{\sigma\in C_{\bft}}sgn(\sigma)\sigma.
$$
Here the sign $sgn(\sigma)$ of a permutation $\sigma$ is defined to be $1$ if $\sigma$ is even and $-1$ if $\sigma$ is odd. Set
$$
c_{\bft}=b_{\bft}a_{\bft}.
$$
The following result is well known (see \cite{Ja}). For each fixed
$\lambda$-tableau $ \bft$, the left ideal  $ \C[S_N] a_\bft$ is isomorphic
to the Young submodule $ M^\lambda$ as described in \cite[4.2]{Ja} (we
consider left $\C[S_N]$-modules here). The Specht module (thus simple
module) $ S^\lambda$ is isomorphic to a left ideal generated by
$ b_\bft a_{\bft} $ ({\cite[4.5]{Ja}}).
Since $\C[S_{N}]$ is semi simple, it follows that
\[ \C[S_{N}]=\oplus _{\lambda} I^\lambda.\]
Here $ I^\lambda$ is the minimal two-sided ideal of $ \C[S_N]$ containing a left ideal which is isomorphic to the simple module $ S^\lambda$, by using
\begin{lem}{\cite[4.5]{Ja}} \label{sym1} For any $\lambda$-tableau $\bft$,
$c_{\bft}$ generates the  minimal ideal $I^\lambda$ of $\C[S_{N}]$.
\end{lem}

We note that $S_N$ is a Coxeter group with generating set $\{ s_1, \cdots, s_{N-1}\}$.  We now fix a (standard) $\lambda$-tableau $ \bft^\lambda$ for which the numbers $\{1, 2, \cdots, N\}$ are placed into the Young diagram $ \lambda$ in the order of top-down in each column starting from the first column. Then the subgroup $ C_{\bft^\lambda}$ is a standard parabolic subgroup of $S_N$.  Note that $ C_{\bft^\lambda}$ is also a Coxeter group.

 For any Coxeter group $(W, S)$, each element $\sigma\in W$ can be written as a product of a minimal number of generating elements in $S$. We denote by $\ell(\sigma)$ the minimal length.

 For any parabolic subgroup $W_I $  generated by elements in $ I \subseteq S$, let $ W^I$ be the minimal coset representatives such that each element $ w \in W$ is of the form $ \tau \rho$ with $ \tau \in W^I$, $ \rho \in W_I$ and $ \ell(\sigma)=\ell(\tau)+\ell(\rho)$. If $W$ is finite, then we have in $\C[W]$,
\begin{equation} \label{product}
\sum_{w\in W} (-1)^{\ell(w)}w=(\sum_{\tau \in W^I}(-1)^{\ell(\tau)}\tau) (\sum_{\rho \in W_I}(-1)^{\ell(\rho)}\rho).
\end{equation}

Recall that for any partition  $\lambda=(\lambda_{1}\geq\cdots\geq\lambda_{k}>0)$, $k=l(\lambda)$ is called the length or  the number of parts of $ \lambda$.
\begin{lem}\label{sym2} If $\lambda$ has at least three parts,  then $I^{\lambda}\subseteq J^{N}$. In particular, $J^N\supseteq \oplus_{\lambda, l(\lambda)\geq 3} I^\lambda $.
\end{lem}
\pf Since $k\geq 3$, then $\{ 1, 2, 3\}$ appears in the first column of $ \bft^\lambda$. Let $ \mu=(N-2,1,1)$ a partition of $ N$ with exactly three parts. The standard tableau $\bft^\mu$ has $\{1,2,3\}$ in the first column and the rest appearing in the first row. Note that
\[C_{\bft^\mu}=S_{3}\times 1\times \cdots \times 1\]
is a parabolic subgroup of $ C_{\bft^{\lambda}}$.  Applying \eqref{product} to the pair $W=C_{\bft^{\lambda}}$ and $ W_I=C_{\bft^\mu}$ we see that $ b_{\bft^{\lambda}}=d b_{\bft^{\mu}}$ for some $ d \in \C[C_{\bft^\lambda}]$.
This shows that  $c_{\bft^{\lambda}}$ is contained in the 2-sided ideal $\langle b_{\bft^{\mu}} \rangle $ in $ \C[S_N]$.  Now the lemma will follow if we show that $ b_{\bft^\mu}\in J^N$. In fact, let $e=s_1+s_2+s_1s_2 s_1=(12)+(23)+(13)$. Here we are writing the element in $S_N$ using cycle decompositions. A direct computation shows that
\[ e(e-3)=3[1+(123)+(132)-(12)-(23)-(13)]=3 b_{\bft^\mu}.\]
Thus $ b_{\bft^\mu} \in J^N$.
\qed

\begin{prop}\label{coro1}
For $N=2M$ or $2M+1$,  $\C[S_{N}]/J^{N}$ has at most  $M+1$ inequivalent irreducible modules.
\end{prop}

\begin{proof} By the lemma above, the only possible irreducible representations of $ \C[S_{N}]/J^{N}$ are those $ S^\lambda$ with $\lambda=(\lambda_1\geq \lambda_2)$,  which is determined by $ \lambda_2=0, \cdots, [\frac{N}{2}]=M$.
\end{proof}

The following lemma follows from the rationality of the Virasoro vertex operator algebra $L(\frac{1}{2},0)$ and the fact that $L(\frac{1}{2},0)$ has three inequivalent irreducible modules
$L(\frac{1}{2},0)$, $L(\frac{1}{2},\frac{1}{2})$ and $L(\frac{1}{2},\frac{1}{16})$.
\begin{lem}\label{Vira-semi}
Let $W$ be an $A(M^{(n)})$-module, then for $0\leq i<j\leq n$, $[\omega^{ij}]([\omega^{ij}]-\frac{1}{2})([\omega^{ij}]-\frac{1}{16})=0$ on $W$.
\end{lem}

For $\delta\in\Z_{2}^{n}$, let $M^{\delta}(k)$, $0\leq k\leq n+1$ with $|\delta'|-k\in 2\Z
$ be the irreducible modules of $M^{(n)}$ given in the above section ( also see \cite{LS}). We denote by $W^{\delta}(k)$ the irreducible module of the Zhu algebra $A(M^{(n)})$ associated to $M^{\delta}(k)$. Recall from (\ref{de}) that $M^{\delta}(k)=M^{\delta'}(k)$. So we also denote $W^{\delta'}(k)=W^{\delta}(k)$.

By \cite{KLY} ( see also \cite{KMY} and \cite{DLMN}), for $n=3$, $M^{(3)}$ is isomorphic to $V_{\sqrt{2}A_{2}}^{+}$. Then  by \cite{DJL} and \cite {ADL}, $M^{(3)}$ is rational and has 20 inequivalent irreducible modules which are  given through $M^{\delta'}(k)$. Furthermore, we have
\begin{lem}\label{module-2} Let $n=3$. Then

(1) If $\delta'=(0,0,0,0)$, then on $W^{\delta'}(k)$ with $k=0,2,4$ we have
$$
[\omega^{ij}]([\omega^{ij}]-\frac{1}{2})=0,$$
for $0\leq i<j\leq 3$.

(2) If $\delta'=(a_{0},\cdots,a_{3})\in\Z_{2}^{4}$ with $|\delta|=1$ and $a_{r}=1$ for some $0\leq r\leq 3$, then on $W^{\delta'}(k)$ with $k=1,3$, we have
$$[\omega^{ri}]=\frac{1}{16}, \ [\omega^{ij}]([\omega^{ij}]-\frac{1}{2})=0,$$
for $0\leq i\neq j\leq 3$ with $i\neq r\neq j$.

(3) If $\delta'=(a_{0},\cdots,a_{3})\in\Z_{2}^{4}$ with $|\delta'|=2$ and $a_{i}=a_{j}=1$ for some $0\leq i<j\leq 3$. Then we have

(a) \ $W^{\delta'}(0)$ is one-dimensional  with basis $e^{\frac{\al^i-\al^j}{2}}-e^{\frac{-\al^i+\al^j}{2}}$ and  on $W^{\delta'}(0)$, we have
$$[\omega^{ij}]=\frac{1}{2}, \ [\omega^{kl}]=0, \ [\omega^{pq}]=\frac{1}{16}; \
$$

(b) \ $W^{\delta'}(2)$ is one-dimensional  with basis $e^{\frac{\al^i+\al^j}{2}}+e^{\frac{-\al^i-\al^j}{2}}$ and  on $W^{\delta'}(2)$, we have
$$[\omega^{ij}]=0, \ [\omega^{kl}]=0, \ [\omega^{pq}]=\frac{1}{16}; \
$$

(c) \ $W^{\delta'}(4)$ is one-dimensional  with basis $(e^{\frac{\al^i+\al^j}{2}+\al^k}+e^{\frac{-\al^i-\al^j}{2}-\al^k})-(e^{\frac{\al^i+\al^j}{2}+\al^l}+e^{\frac{-\al^i-\al^j}{2}-\al^l})$ and  on $W^{\delta'}(4)$, we have
$$[\omega^{ij}]=0, \ [\omega^{kl}]=\frac{1}{2}, \ [\omega^{pq}]=\frac{1}{16};
$$
where $k,l\in\{0,1,2,3\}\setminus\{i,j\}$ with $k\neq l$, $p\in\{i,j\}, q\in\{k,l\}$.
\end{lem}

More generally we have
\begin{lem}\label{module-1} (1) If $\delta'=(0,\cdots,0)\in \Z_{2}^{n+1}$, $0\leq 2k\leq n+1$ and $0\leq i<j\leq n$, then on $W^{\delta'}(2k)$,
$$
[\omega^{ij}]([\omega^{ij}]-\frac{1}{2})=0.
$$
(2) If $\delta'=(a_{0},a_{1},\cdots,a_{n})\in\Z_{2}^{n+1}$ such that $a_{i_{1}}=a_{i_{2}}=\cdots =a_{i_N}=1$ and $|\delta'|=N$ for some $1\leq N\leq \frac{n+1}{2}$ and $0\leq i_{1}<\cdots <i_{N}\leq n$. Then on $W^{\delta'}(N-2k)$ with  $0\leq N-2k\le N$,
$$
[\omega^{ij}]=\frac{1}{16},  \ [\omega^{pq}]([\omega^{pq}]-\frac{1}{2})=0, \ [\omega^{kl}]=0,
$$
and on $W^{\delta'}(N+2k)$ with $N<N+2k\leq n+1$,
$$
[\omega^{ij}]=\frac{1}{16}, \ [\omega^{pq}]=0, \ [\omega^{kl}]([\omega^{kl}]-\frac{1}{2})=0,
$$
where $i\in\{i_{1},i_{2},\cdots,i_{N}\}$, $j\in \{0,1,\cdots,n\}\setminus \{i_{1},i_{2},\cdots,i_{N}\}$, $p,q\in\{i_{1},\cdots,i_{N}\}$ with $p\neq q$,  $k,l\in \{0,1,\cdots,n\}\setminus \{i_{1},\cdots,i_{N}\}$ with $k\neq l$.
\end{lem}

Let $I^{n}$ be the two-sided ideal of $A(M^{(n)})$  generated by $[\omega^{ij}]([\omega^{ij}]-\frac{1}{2})$, $0\leq i<j\leq n$. We have the following result.
\begin{lem}\label{sym3}
For $n+1=2M$ or $2M+1$, $A(M^{(n)})/I^{n}$ has at least $M+1$ inequivalent irreducible modules.
\end{lem}

\pf \ We  consider the decomposition:
  $$V_{\Z\alpha^0}\otimes V_{\Z\alpha^1}\otimes\cdots\otimes V_{\Z\alpha^n}=\sum\limits_{0\leq 2k\leq n+1}L_{\widehat{\frak g}}(n+1,2k)\otimes M^{\underline{0}}(2k).$$
    Then  it is easy to check that $W^{\underline{0}}(2k)$, $0\leq 2k\leq n+1$,  $\underline{0}=(0,0,\cdots,0)\in\Z_{2}^{n}$ are inequivalent irreducible modules of $A(M^{(n)})/I^{n}$.
    \qed

We are now in a position to state the following lemma.
\begin{lem}\label{l4.3}
$A(M^{(n)})/I^{n}$ is isomorphic to  the  associative algebra $T^{n+1}$ and $W^{\underline{0}}(2k)$, $0\leq 2k\leq n+1$,  $\underline{0}=(0,0,\cdots,0)\in\Z_{2}^{n}$ are all the inequivalent irreducible modules of $A(M^{(n)})/I^{n}$.
\end{lem}
\pf \ For $[u]\in A(M^{(n)})$, denote its image in $A(M^{(n)})/I$ by $\overline{[u]}$.  Let
$$
\lambda^{ij}=1-4\omega^{ij},  \ \ 0\leq i<j\leq n.
$$
Then for $0\leq i\neq j\leq n$, by the fact that $\overline{[\omega^{ij}]}(\overline{[\omega^{ij}]}-\frac{1}{2})=0$,  we have
$$(\overline{[\lambda^{ij}]})^2=1.
$$
 Let $s^i=\overline{[\lambda^{i-1,i}]}$, $i=1,2,\cdots,n$.  By (\ref{zhu-sym}) and Lemma \ref{l4.1}, we have
  $$
  s^is^j=s^js^i, \ \ |i-j|\geq 2,
  $$
  $$
  s^is^{i+1}s^i=s^{i+1}s^is^{i+1}=\overline{[\lambda^{i-1,i+1}]}, \ \ i=1,2,\cdots,n-1.
  $$
  By Lemma \ref{l4.1}, it is easy to check that
  $$
 (s^{i}+s^{i+1}+s^{i+1}s^{i}s^{i+1})(s^{i}+s^{i+1}+s^{i+1}s^{i}s^{i+1}-3)=0, 1\leq i\leq n-1.
  $$
  Define $\psi: \C[S_{n+1}]\rightarrow A(M^{(n)})/I$ by
  $$
  s_{i+1,j+1} \mapsto \overline{[\lambda^{ij}]}. \
  $$
  In particular,
  $$
  s_i\mapsto s^{i}.
  $$
  Then $A(M^{(n)})/I^{n}$ is isomorphic to a quotient of $\C[S_{n+1}]/J^{n+1}$. The lemma follows from Proposition \ref{coro1} and Lemma \ref{sym3}.
  \qed

Let $N\in\Z$ such that $1\leq N\leq n$. For $0\leq i_{1}<i_{2}<\cdots <i_{N}\leq n$. Denote
$$\{i_{1},i_{2},\cdots,i_{N}\}^c=\{0,1,2,\cdots,n\}\setminus \{i_{1},i_{2},\cdots,i_{N}\}.$$
Let $I(i_{1},i_{2},\cdots,i_{N})$ be the two-sided ideal of $A(M^{(n)})$ generated by $[\omega^{ij}]-\frac{1}{16},$ $[\omega^{kl}]([\omega^{kl}]-\frac{1}{2})$, where $i\in\{i_{1},i_{2},\cdots,i_{N}\}$, $j\in \{i_{1},i_{2},\cdots,i_{N}\}^c$, $k<l$, and $k,l\in\{i_{1},\cdots,i_{N}\}$ or $k,l\in \{i_{1},\cdots,i_{N}\}^c$. Let $P(N)$ be the complex commutative polynomial algebra with variables $x^{ij}$, $i=1,2,\cdots, N; j=N+1,N+2,\cdots,n+1$ and $\bar{P}(N)$ the quotient algebra of $P(N)$ by the ideal generated by $x^{ij}-\frac{1}{16}$, where  $i=1,2,\cdots, N; j=N+1,N+2,\cdots,n+1$.
\begin{remark}
Obviously, we may assume that $1\leq N\leq \frac{n+1}{2}$.
\end{remark}
We have the following lemma.

\begin{lem}\label{l4.4}
The associative algebra $A(M^{(n)})/I(i_{1},i_{2},\cdots,i_{N})$ is isomorphic to a quotient of  the associative algebra $T^{N}\otimes T^{n+1-N}\otimes \bar{P}(N)$.
\end{lem}
 \pf \ For $[u]\in A(M^{(n)})$, we denote its image in $A(M^{(n)})/I(i_{1},i_{2},\cdots,i_{N})$ by $\overline{[u]}$.  Define a linear map $\phi: \C[S_{N}]\otimes \C[S_{n+1-N}]\otimes {P}(N)\rightarrow A(M^{(n)})/I(i_{1},\cdots,I_{N})$ by
$$
s_{pq}\otimes 1\otimes 1 \mapsto \overline{[\omega^{i_{p}i_{q}}]}, \ \ 1\otimes s_{kl}\otimes 1\mapsto \overline{[\omega^{i_{N+k}i_{N+l}}]}, \ \ 1\otimes 1\otimes x^{rm}\mapsto \overline{[\omega^{i_{r}i_{m}}]},$$
for $p,q=1,2,\cdots,N, p<q$, $k,l=1,2,\cdots, n+1-N, k<l$, $r=1,2,\cdots,N, m=N+1, N+2,\cdots,n+1$.
It is easy to check that $\phi$ is an algebra homomorphism and
$$
\phi(J^{N}\otimes 1\otimes 1)=\phi(1\otimes J^{n+1-N}\otimes 1)=0, \ \phi(1\otimes 1\otimes  (x^{ij}-\frac{1}{16}))=0.
$$
The lemma follows. \qed

\begin{lem}\label{l4.2}
Let $W$ be an irreducible module of $A(M^{(n)})$. For fixed $0\leq i<j\leq n$, if there exists $u\in W$ such that $[\omega^{ij}]u=\frac{1}{16}u$, then $[\omega^{ij}]$ acts on $W$ as $\frac{1}{16}$.
\end{lem}
\pf \ By Lemma \ref{l4.1} we may assume that
$$
W=W^1\oplus W^2
$$
such that $[\omega^{ij}]u=\frac{1}{16}u$, for $u\in W^1$, and $([\omega^{ij}])([\omega^{ij}]-\frac{1}{2})=0$ on $W^2$. We will prove that $W^1$ is an $A(M^{(n)})$-submodule of $W$. By Theorem \ref{p4.2}, it suffices to prove that $W^1$ is invariant under actions of $[\omega^{ij}]$, for all $0\leq i<j\leq n$.  By (\ref{zhu-sym}), $W^1$ is invariant under the action of $[\omega^{kl}]$ for $0\leq k,l\leq n$ such that $k,l\neq i,j$. We now consider $\omega^{jl},\omega^{il}$, for $0\leq l\leq n$. Decompose $W$ into direct sum of irreducible
modules of the algebra $A^{(ijl)}=\C[\omega^{ij}]\oplus \C[\omega^{jl}]\oplus \C[\omega^{il}]$. By Lemma \ref{l4.1}, $W^1$ is a direct sum of irreducible one-dimensional $A^{(ijl)}$-modules such that $[\omega^{ij}]$ acts as $\frac{1}{16}$. So $W^1$ is invariant under $A^{(ijl)}$. This deduces that $W^1$ is a module of $A(M^{(n)})$. Since $W$ is irreducible, it follows that $W^1=W$. We complete the proof. \qed

\begin{lem}\label{irr}
Let $W$ be an irreducible module of $A(M^{(n)})$. If there exists  $\omega^{pq}$ such that $[\omega^{pq}]([\omega^{pq}]-\frac{1}{2})W\neq 0$. Then there exist $1\leq N\leq \frac{n+1}{2}$ and  $0\leq i_{1}<i_{2}<\cdots<i_{N}\leq n$ such that $W$ is an irreducible module of $A(M^{(n)})/I(i_{1},\cdots,i_{N})$.
\end{lem}
{\bf Proof} \ \ By the assumption and Lemma \ref{Vira-semi},   $[\omega^{pq}]$ has an eigenvector with eigenvalue $\frac{1}{16}$ in $W$. By Lemma \ref{l4.2}, $[\omega^{pq}]=\frac{1}{16}$ on $W$. For $0\leq r\leq n$ different from $p$ and $q$, by Lemma \ref{l4.1}, Lemma \ref{Vira-semi} and Lemma \ref{l4.2}, $[\omega^{pr}]=\frac{1}{16}$ and $[\omega^{qr}]([\omega^{qr}]-\frac{1}{2})=0$ or $[\omega^{qr}]=\frac{1}{16}$ and $[\omega^{pr}]([\omega^{pr}]-\frac{1}{2})=0$.
 Then there exist $N\geq 1$ and  distinct $0\leq i_{1},i_{2},\cdots,i_{N}, j_{1},j_{2},\cdots,j_{n+1-N}\leq n$ such that
$$
[\omega^{pi_{k}}]=\frac{1}{16}, \ [\omega^{qj_{s}}]=\frac{1}{16}, \ k=1,2,\cdots,N, s=1,2,\cdots,n+1-N.
$$
Obviously, $p\in\{j_{1},\cdots,j_{n+1-N}\}$, $ q\in\{i_{1},i_{2},\cdots,i_{N}\}$ and $\{j_{1},\cdots,j_{n+1-N}\}=\{i_{1},\cdots,i_{N}\}^c$. So we have
$$
[\omega^{qi_{k}}]([\omega^{qi_{k}}]-\frac{1}{2})=0, \ [\omega^{pj_{s}}]([\omega^{pj_{s}}]-\frac{1}{2})=0, \ k=1,2,\cdots,N, s=1,2,\cdots,n+1-N,
$$
and therefore
$$
[\omega^{i_{r}j_{s}}]=\frac{1}{16}, \ r=1,2,\cdots,N, s=1,2,\cdots,n+1-N.
$$
Note that for $1\leq k\neq l\leq N$,
$$
[\omega^{qi_{k}}]([\omega^{qi_{k}}]-\frac{1}{2})=[\omega^{qi_{l}}]([\omega^{qi_{l}}]-\frac{1}{2})=0,
$$
so for $1\leq k\neq l\leq N$, we have
$$
[\omega^{i_{k}i_{l}}]([\omega^{i_{k}i_{l}}]-\frac{1}{2})=0.
$$
Similarly, for $1\leq k\neq l\leq n+1-N$, we have
$$
[\omega^{j_{k}j_{l}}]([\omega^{j_{k}j_{l}}]-\frac{1}{2})=0.
$$
We prove that $W$ is an $A(M^{(n)})/I(i_{1},\cdots,i_{N})$-module. \qed

\begin{defn} Let $W$ be an $A(M^{(n)})$-module. If for any $0\leq i\neq j\leq n$, $[\omega^{ij}]([\omega^{ij}]-\frac{1}{2})=0$ on $W$, then $W$ is called a type-(I) module. If there exist $1\leq N\leq n$ and  $0\leq i_{1}<i_{2}<\cdots<i_{N}\leq n$ such that $W$ is an $A(M^{(n)})/I(i_{1},\cdots,i_{N})$-module, then we call $W$ a type-(II) module.
\end{defn}

\vskip 0.3cm
  For $1\leq M\leq n$ and $0\leq j_{1}<\cdots <j_{M}\leq n$,
   let  $I^{\{j_{1},\cdots,j_{M}\}}$ be the two-sided ideal of $A(M^{(n)})$ generated by $[\omega^{kl}]$, $[\omega^{pq}]-\frac{1}{16}$ and $[\omega^{ij}]([\omega^{ij}]-\frac{1}{2})$,  $p,i,j\in\{j_{1},\cdots,j_{M}\}, i\neq j$, $q,k,l\in\{j_{1},\cdots,j_{M}\}^{c}=\{0,1,\cdots,n\}\setminus \{j_{1},\cdots,j_{M}\}, k\neq l$. Then we have
    \begin{lem}\label{module-4}
    Let $1\leq M\leq n$ and $0\leq j_{1}<\cdots <j_{M}\leq n$. Then there are exactly $l+1$ irreducible $A(M^{(n)})/I^{\{j_{1},\cdots,i_{M}\}}$-modules  which are the lowest weight spaces of $M^{\delta'}(M-2k)$, $0\leq 2k\leq M$, where $\delta'=(a_{0},\cdots,a_{n})$ such that $a_{k}=1$ for  $k\in\{j_{1},\cdots,j_{M}\}$ and $a_{k}=0$ for $k\in\{j_{1},\cdots,j_{M}\}^c$, where $M=2l$ or $M=2l+1$.
    \end{lem}
  \pf   Let $A^{(j_{1},\cdots,j_{M})}$ be the subalgebra of $A(M^{(n)})$ generated by $[\omega^{ij}], i,j\in\{j_{1},\cdots,j_{M}\}$ and $I'^{\{j_{1},\cdots,j_{M}\}}$ the two-sided ideal of $A^{(j_{1},\cdots,j_{M})}$ generated by $[\omega^{ij}]([\omega^{ij}]-\frac{1}{2})$, $i,j\in\{j_{1},\cdots,j_{M}\}$.

   Let $W$ be an irreducible $A(M^{(n)})/I^{\{j_{1},\cdots,i_{M}\}}$-module,
  then  $W$ is also  an irreducible $A^{(j_{1},\cdots, j_{M})}/I'^{\{j_{1},\cdots,j_{M}\}}$-module. So by Lemma \ref{l4.3}, there are at most $l+1$ inequivalent irreducible $A(M^{(n)})/I^{\{j_{1},\cdots,i_{M}\}}$-modules. Consider the decomposition of
  $$
  V_{\Z\al^0+\frac{a_0}{2}\al^0}\otimes  V_{\Z\al^1+\frac{a_1}{2}\al^1}\otimes\cdots  V_{\Z\al^n+\frac{a_n}{2}\al^n}
  $$
  such that $a_{k}=1$ for $k\in\{j_{1},\cdots,j_{M}\}$ and $a_{k}=0$ for $k\in\{j_{1},\cdots,j_{M}\}^c$. Then there are exactly $l+1$ inequivalent irreducible $A(M^{(n)})/I^{\{j_{1},\cdots,i_{M}\}}$-modules which are the lowest weight spaces of $M^{\delta'}(M-2k)$, $0\leq 2k\leq M$. The lemma follows. \qed

We have the following lemma.
\begin{lem}\label{module-3}
If $W$ is an irreducible type-(II) module of $A(M^{(n)})$. Then there exist $1\leq N\leq \frac{n+1}{2}$ and $0\leq i_{1}<\cdots <i_{N}\leq n$ such that $W$ is an irreducible $A(M^{(n)})/I^{\{i_1,\cdots,i_N\}}$-module or $W$ is an irreducible $A(M^{(n)})/I^{\{i_1,\cdots,i_N\}^c}$-module.
\end{lem}
\pf By the assumption, there exist $1\leq N\leq \frac{n+1}{2}$ and $0\leq i_{1}<\cdots <i_{N}\leq n$ such that
$$[\omega^{pq}]=\frac{1}{16}, \ [\omega^{kl}]([\omega^{kl}]-\frac{1}{2})=0, $$
for $p\in\{i_{1},i_{2},\cdots,i_{N}\}$, $q\in \{i_{1},i_{2},\cdots,i_{N}\}^c$,  and $k,l\in\{i_{1},\cdots,i_{N}\}$, or $k,l\in \{i_{1},\cdots,i_{N}\}^c$ with $k\neq l$.

 Obviously, if $N=1$, then $W$ is an irreducible $A(M^{(n)})/I^{\{i_{1}\}^c}$-module.

  We now assume that $2\leq N\leq \frac{n+1}{2}$. If for any $i,j\in\{i_{1},\cdots,i_{N}\}$ with  $i\neq j$, $[\omega^{ij}]=0$ on $W$, then $W$ is an irreducible module of $A(M^{(n)})/I^{\{i_{1},\cdots,i_{N}\}^c}$. If there exist $i,j\in\{i_{1},\cdots,i_{N}\}$ with $i\neq j$ such that $[\omega^{ij}]\neq 0$ on $W$. Note that on $W$ $[\omega^{ij}]([\omega^{ij}]-\frac{1}{2})=0$, then
  $$
  U=\{u\in W|[\omega^{ij}]u=\frac{1}{2}u\}\neq 0.
  $$
 For any $k,l\in\{i_{1},\cdots,i_{N}\}^c$ with $k\neq l$, since $[\omega^{kl}][\omega^{ij}]=[\omega^{ij}][\omega^{kl}]$, it follows that $U$ is invariant under the action of $[\omega^{kl}]$. Let $A^{(ijkl)}$ be the subalgebra of $A(M^{(n)})$ generated by all $[\omega^{pq}]$, $p,q\in\{i,j,k,l\}$ with $p\neq q$. Note that on $W$, $[\omega^{pq}]=\frac{1}{16}$, for $p=i,j, q=k,l$.
 Then $U$ is an $A^{(ijkl)}$-module. Since $A^{(ijkl)}$ is the Zhu algebra of $M^{(3)}$, it follows that $A^{(ijkl)}$ is semisimple and $U$ is a direct sum of irreducible $A^{(ijkl)}$-modules. Then by (3) of Lemma \ref{module-2}, we have $[\omega^{kl}]=0$ on $U$. We prove that for any $k,l\in\{i_{1},\cdots,i_{N}\}^c$ with $k\neq l$, $$[\omega^{kl}]|_{U}=0.$$
  Let $A^{(i_{1},i_{2},\cdots,i_{N})}$ be the subalgebra of $A(M^{(n)})$ generated by $[\omega^{rs}]$, $r,s\in\{i_{1},\cdots,i_{N}\}$ with $r\neq s$.
 Let $W(U)\subseteq W$ be the $A^{(i_{1},\cdots,i_{N})}$-submodule of $W$ generated by $U$. Since for any $r,s\in\{i_{1},\cdots,i_{N}\}$, $k,l\in\{i_{1},\cdots,i_{N}\}^c$, $[\omega^{rs}][\omega^{kl}]=[\omega^{kl}][\omega^{rs}]$, it follows that for any  $k,l\in\{i_{1},\cdots,i_{N}\}^c$,
 $$
 [\omega^{kl}]|_{W(U)}=0, $$
 proving that $W(U)$ is an $A(M^{(n)})$-module. Since $W$ is irreducible, we have $W=W(U)$. We show that $W$ is an irreducible $A(M^{(n)})/I^{(i_{1},\cdots,i_{N})}$-module.
\qed

We are now in a position to state the main result of this section.
\begin{theorem}\label{th4.1}
$M^{(n)}$ has $2^{n-1}(n+2)$ inequivalent irreducible modules which are given by
\begin{equation}\label{module1}
M^{\delta}(2k), \ \delta\in\Z^{n}_{2}, \ 0\leq 2k\leq n+1,
\end{equation}
if $n$ is even and
\begin{equation}\label{module2}
M^{\delta}(k), \ 0\leq k\leq n+1, \ \delta\in\Z^{n}_{2}, \ {\rm with} \ |\delta|\equiv k (mod 2),
\end{equation}
if $n$ is odd.
\end{theorem}

 \pf  Let $M$ be an irreducible module of $M^{(n)}$. Since $M^{(n)}$ is $C_{2}$-cofinite, $M=\oplus_{m=0}^{\infty}M_{(m)}$ is an ordinary module of $M^{(n)}$. Then the lowest weight space $W=M_{(0)}$ is an irreducible module of $A(M^{(n)})$.  By Lemma \ref{irr}, $W$ is a type-(I) or type-(II) module. If $W$ is a module of type-(I), then $W$ is an irreducible module of $A(M^{(n)})/I^{n}$. By Lemma \ref{l4.3}, $M$ is isomorphic to one of $W^{\underline{0}}(2k), 0\leq 2k\leq n+1$. So $M$ is isomorphic to one of $M^{\underline{0}}(2k), 0\leq 2k\leq n+1$. Then we may assume that $W$ is a type-(II) module of $A(M^{(n)})$. So there exist $N\geq 1$ and $0\leq i_{1}<i_{2}<\cdots <i_{N}\leq n$ such that $W$ is an irreducible $A(M^{(n)})/I(i_{1},\cdots,i_{N})$-module. Then the theorem follows from Lemma \ref{module-3} and Lemma \ref{module-4}. \qed

\section{Rationality   of $M^{(n)}$}
\def\theequation{5.\arabic{equation}}
\setcounter{equation}{0}

It is already known that $M^{(n)}$ is $C_{2}$-cofinite. In this section, we will prove that $M^{(n)}$ is rational. We first have the following result.

\begin{theorem}\label{th5.1}
The Zhu algebra $A(M^{(n)})$ is semisimple.
\end{theorem}
 \pf Let $W$ be  an $A(M^{(n)})$-module. If $[\omega^{ij}]([\omega^{ij}]-\frac{1}{2})=0$, for all $0\leq i<j\leq n$, then $W$ is a module for $A(M^{(n)})/I^{n}$. By Lemma \ref{l4.3}, $W$ is completely reducible. Now assume that there exists $\omega^{pq}$ such that $[\omega^{pq}]([\omega^{pq}]-\frac{1}{2})\neq 0$. Then there exists $0\neq u\in W$ such that $[\omega^{pq}]u=\frac{1}{16}u$. We may assume
$$
W=W^{1}\oplus (W^2+W^3),
$$
such that $$[\omega^{pq}]|_{W^1}=\frac{1}{16}, \ [\omega^{pq}]|_{W^2}=0, \ [\omega^{pq}]|_{W^3}=\frac{1}{2}.$$
As proof in Lemma \ref{l4.2}, we have $W^1$ is an $A(M^{(n)})$-module. We now prove that $W^2+W^3$ is also an $A(M^{(n)})$-module. It suffices to prove that for any $0\leq l\leq n$, $l\neq p,q$, $W^2+W^3$ is an module of $A^{(pql)}=\C[\omega^{pq}]+\C[\omega^{pl}]+\C[\omega^{ql}]$. We know that $W$ is a direct sum of irreducible one-dimensional or two-dimensional modules of $A^{(pql)}$. Let $U$ be an irreducible $A^{(pql)}$-submodule of $W$.  If $U$ is one dimensional, then there exits $u^i\in W^i$, $i=1,2,3$ such that $U=\C u=\C(u^1+u^2+u^3)$ and $\omega^{pq}(u^1+u^2+u^3)=\lambda(u^1+u^2+u^3)$, where $\lambda=\frac{1}{16},0,$ or $\frac{1}{2}$. It follows that $u=u^1$ , $u=u^2$, or $u=u^3$. Then $U\subseteq W^1$ or $U\subseteq W^2+W^3$. If $U$ is two-dimensional, then there exist $u^i,v^i\in W^i$, $i=1,2,3$ such that $u=u^1+u^2+u^3, v=v^1+v^2+v^3$ is a basis of $U$ and the matrix of $[\omega^{pq}]$ under this basis is
$$
\left[\begin{array}{ccc} 0 & 0\\
0 & \frac{1}{2}
\end{array}\right].
$$
Then $$[\omega^{pq}](u^1+u^2+u^3)=0, \ \ [\omega^{pq}](v^1+v^2+v^3)=\frac{1}{2}(v^1+v^2+v^3).$$
It follows that $u=u^2$, $v=v^3$. Therefore $U\subseteq W^2+W^3$. We prove that $W^2+W^3$ is an $A(M^{(n)})$-module.  We now consider $W^1$ and $W^2+W^3$ respectively. We may decompose $W^1$ and $W^2+W^3$ as we do for $W$. Continuing the process, we can decompose $W$ into direct sum of $A(M^{(n)})$-modules $M^1,\cdots,M^r$ such that for each $i$, $M^i$ is a type-(I) module or a type (II)-module. By Lemma \ref{l4.3} and Lemma \ref{l4.4}, we know that $M^i$, $i=1,2,\cdots,r$ are completely reducible.
\qed

 Let $M$ be an irreducible module of $M^{(n)}$. By the discussion in the above section, the lowest weight space $W$ of $M$ is a type-(I) or  type-(II) module of $A(M^{(n)})$. We correspondingly call $M$ a type-(I) or type-(II) $M^{(n)}$-module. We first  have the following lemma.
\begin{lem}\label{l6.1}
Let $M^1$ and $M^{2}$ be irreducible  $M^{(n)}$-modules of different type. Then $$Ext_{M^{(n)}}^1(M^2, M^1) = 0.$$
\end{lem}
\pf We first assume that $M^1$ is a tpye-(I) irreducible $M^{(n)}$-module and $M^2$ is a type-(II) irreducible $M^{(n)}$-module. Let $W^1$ and $W^2$ be the  irreducible $A(M^{(n)})$-modules corresponding to $M^1$ and $M^2$ respectively. Note that for any $0\leq i<j\leq n$,
\begin{equation}\label{type1}
[\omega^{ij}]([\omega^{ij}]-\frac{1}{2})|_{W^1}=0.
\end{equation}
On the other hand, there exists $[\omega^{pq}]\in A(M^{(n)})$ such that
\begin{equation}\label{type2}
[\omega^{pq}]|_{W^2}=\frac{1}{16}.
\end{equation}
For convenience, without loss of generality, we may assume that $p=0,q=1$. 
Let $L^{(01)}(\frac{1}{2},0)$ be the Virasoro vertex operator algebra generated by $\omega^{01}$. Note that $M^{(n)}$ contains a vertex operator subalgebra $L=L(c_1,0)\otimes L(c_2,0)\otimes \cdots\otimes L(c_n,0)$ with the same Virasoro element as $M^{(n)}$. Here  $L(c_1,0)=L^{(01)}(\frac{1}{2},0)$.
By Lemma \ref{l3.1} (also see Lemma \ref{l4.10}), Lemma \ref{l4.9} (also see Lemma \ref{module-1}), Theorem \ref{n2.2c} and fusion rules of $L^{(01)}(\frac{1}{2},0)$ (see Theorem \ref{wang}), $M^{(n)}$ is a direct sum of irreducible $L^{(01)}(\frac{1}{2},0)$-submodules isomorphic to $L(\frac{1}{2},0)$ or $L(\frac{1}{2},\frac{1}{2})$. By (\ref{type1})-(\ref{type2}), Theorem \ref{n2.2c} and Theorem \ref{wang}, $M^1$ is a direct sum of irreducible $L^{(01)}(\frac{1}{2},0)$-submodules isomorphic to $L(\frac{1}{2},0)$ or $L(\frac{1}{2},\frac{1}{2})$ and $M^2$ is a direct sum of irreducible $L^{(01)}(\frac{1}{2},0)$-submodules isomorphic to  $L(\frac{1}{2},\frac{1}{16})$.

Let $M$ be an extension of $M^2$ by $M^1.$  Then  there is a
natural $L$-module embedding $M^2\rightarrow M$. So we may assume
that $M=M^1\oplus M^2$ as $L$-modules as $L$ is rational. Let
$N^1,N,N^2$ be any irreducible $L$-submodules of $M^1,M^{(n)},M^2,$
respectively. Then $P_{N^1}Y(u,z)|_{N^2}$ for $u\in N$ is an
intertwining operator of type $I_{L}\left(\begin{tabular}{c}
$N^1$\\
$N$\  $N^2$
\end{tabular}\right)$, where  $P_{N^1}$ is the projection from $M$ to $N_1.$ Since both $N$ and $N^1$ are direct sums of irreducible $L^{(01)}(\frac{1}{2},0)$-modules isomorphic to $L(\frac{1}{2},0)$ or $L(\frac{1}{2},\frac{1}{2})$, and $N^2$ is a direct sum of irreducible $L(\frac{1}{2},0)$-modules isomorphic to $L(\frac{1}{2},\frac{1}{16})$,
 by Theorem \ref{wang} and Theorem \ref{n2.2c}, $P_{N^1}Y(u,z)|_{N^2}=0.$ Note that $N^1, N, N^2$ are arbitrary, we see that
$u_nM^2\subset M^2$ for any $u\in M^{(n)}$ and $n\in \Z.$ As a result,
$M^2$ is an $M^{(n)}$-module and ${\rm Ext^1}_{M^{(n)}}(M^2, M^1) = 0.$ 

If $M^1$ is a tpye-(II) irreducible $M^{(n)}$-module and $M^2$ is a type-(I) irreducible $M^{(n)}$-module, the proof is similar.
\qed

\begin{lem}\label{l6.2}
Let $M^1$ and $M^2$ be two inequivalent irreducible $M^{(n)}$-modules of type-(I).
Then $Ext_{M^{(n)}}^1(M^2, M^1) = 0.$
\end{lem}
{\bf Proof}  \ Since for $n=2$ and $n=3$, $M^{(n)}$ are rational, we can show the lemma by induction on $n$. Assume that the lemma holds for  $M^{(m)}$, $2\leq m\leq n-1$. We now consider $M^{(m)}$, $m=n$.  By Lemma \ref{l4.3} and Theorem \ref{th4.1}, $M^1=M^{{0^{n}}}(2l_1)$ and $M^2=M^{{0^{n}}}(2l_2)$ for some $0< 2l_1,2l_2\leq n+1$. By Theorem \ref{th5.1}, if two irreducible $M^{(n)}$-modules have the same lowest weight, then the extension of one module by the other  is trivial. So we may assume that $l_1\neq l_2$. Note that
$$V_{\Z\alpha^0}\otimes V_{\Z\alpha^1}\otimes\cdots\otimes V_{\Z\alpha^n}=\sum\limits_{0\leq 2k\leq n+1}L_{\widehat{\frak g}}(n+1,2k)\otimes M^{0^{n}}(2k)$$
and
$$
M^{0^n}(2k)=\oplus_{0\leq 2l\leq n}M^{0^{n-1}}(2l)\otimes L(c_{n}, h^{(n)}_{2l+1,2k+1}),
$$
where $0^{n-1}=(0,0,\cdots,0)\in\Z^{n-1}_{2}$.
So $M^{(n)}$ contains a subalgebra $U=M^{0^{n-1}}(0)\otimes L(c_{n}, 0)$ with the same Virasoro vector. By inductive assumption and the fact that $L(c_{n},0)$ is rational,  $U$ is rational. Furthermore, we have
$$
M^{(n)}=M^{0^n}(0)=\oplus_{0\leq 2l\leq n}M^{0^{n-1}}(2l)\otimes L(c_{n}, h^{(n)}_{2l+1,1}),
$$
$$
M^1=M^{0^n}(2l_1)=\oplus_{0\leq 2l\leq n}M^{0^{n-1}}(2l)\otimes L(c_{n}, h^{(n)}_{2l+1,2l_1+1}),
$$
$$
M^2=M^{0^n}(2l_2)=\oplus_{0\leq 2l\leq n}M^{0^{n-1}}(2l)\otimes L(c_{n}, h^{(n)}_{2l+1,2l_2+1}),
$$
and $M^{0^{n-1}}(2l), 0\leq 2l\leq n$ are irreducible $M^{0^{n-1}}(0)$-modules.
Let $N, N^{1}$ and $N^{2}$ be irreducible $U$-submodules of $M^{(n)}, M^{1}$ and $M^2$ respectively. Since $l_1\neq l_{2}$, by Corollary \ref{wang-lam-co}, for $0\leq 2k,2l,2p\leq n$,
$$I_{L(c_{n},0)}\left(\begin{tabular}{c}
$L(c_{n},h^{(n)}_{2p+1,2l_1+1})$\\
$L(c_n, h^{(n)}_{2k+1,1})$ $L(c_n, h^{(n)}_{2l+1,2l_2+1})$\\
\end{tabular}\right)=0.$$
By Theorem \ref{n2.2c}, we have
$$I_{U}\left(\begin{tabular}{c}
$N^1$\\
$N$\  $N^2$
\end{tabular}\right)=0.$$
Then the lemma follows from Lemma \ref{la1}. \qed

\begin{lem}\label{l6.3}
Let $M^1$ and $M^2$ be two inequivalent irreducible $M^{(n)}$-modules of type-(II).
Then $Ext_{M^{(n)}}^1(M^2, M^1) = 0.$
\end{lem}

 \pf Let $W^1$ and $W^2$ be the lowest weight spaces of $M^1$ and $M^2$ respectively. Then $W^1$ and $W^2$ are two irreducible $A(M^{(n)})$-modules of type-(II). So there exist $1\leq N_1,N_2\leq n$ and  $0\leq i_{1}<i_{2}<\cdots<i_{N_1}\leq n$, $0\leq j_{1}<j_{2}<\cdots<j_{N_2}\leq n$ such that
$$
[\omega^{ij}]=\frac{1}{16}, \ [\omega^{kl}]([\omega^{kl}]-\frac{1}{2})=0
$$
on $W^1$ for $i\in\{i_{1},i_{2},\cdots,i_{N_1}\}$, $j\in \{i_{1},i_{2},\cdots,i_{N_1}\}^c$, $k,l\in\{i_{1},\cdots,i_{N_1}\}$ or $k,l\in \{i_{1},i_{2},\cdots,i_{N_1}\}^c$, and
$$
[\omega^{ij}]=\frac{1}{16}, \ [\omega^{kl}]([\omega^{kl}]-\frac{1}{2})=0
$$
on $W^2$ for $i\in\{j_{1},j_{2},\cdots,j_{N_2}\}$, $j\in \{j_{1},j_{2},\cdots,j_{N_2}\}^c$, $k,l\in\{j_{1},\cdots,j_{N_2}\}$ or $k,l\in \{j_{1},j_{2},\cdots,j_{N_2}\}^c$.
If $\{i_{1},\cdots,i_{N_1}\}\neq \{j_{1},\cdots,j_{N_2}\}$ and $\{i_{1},\cdots,i_{N_1}\}\neq \{j_{1},\cdots,j_{N_2}\}^c$, then there exist $0\leq i_{p},i_{q}\leq n, 1\leq p\neq q\leq N_1$ such that $i_{p}\in\{j_{1},\cdots,j_{N_2}\}^c, i_{q}\in\{j_{1},\cdots,j_{N_2}\}$. So
$$
[\omega^{i_{p}i_{q}}]([\omega^{i_{p}i_{q}}]-\frac{1}{2})|_{W^1}=0.
$$
On the other hand, we have
$$
([\omega^{i_{p}i_{q}}]-\frac{1}{16})|_{W^2}=0.
$$
 As the proof of Lemma \ref{l6.1}, we have ${\rm Ext}_{M^{(n)}}^1(M^2, M^1) = 0$. If $\{i_{1},\cdots,i_{N_1}\}= \{j_{1},\cdots,j_{N_2}\}$ or $\{i_{1},\cdots,i_{N_1}\}=\{j_{1},\cdots,j_{N_2}\}^c$.  By Lemma \ref{lam-s2}, for $\delta,\sigma\in\Z_{2}^{n}$, $1^n=(1,1,\cdots,1)\in\Z_{2}^n$ and $0\leq k,l\leq n+1$, $M^{\delta}(k)\cong M^{\sigma}(l)$ if and only if either $\delta=\sigma$ and $k=l$ or $l=n+1-k$ and $\sigma=\delta+1^{n}$.
 Then by Lemma \ref{module-4} and Lemma \ref{module-3}, we may assume that $\{i_{1},\cdots,i_{N_1}\}= \{j_{1},\cdots,j_{N_2}\}$ and $M^1\cong M^{{\delta'}}(l_1)$ and $M^2\cong M^{{\delta'}}(l_2)$, for some $0\leq l_1\neq l_2\leq n$ such that $l_1-N_1\in 2\Z$, $l_2-N_1\in2\Z$, where ${\delta'}=(a_{0},\cdots,a_{n})$ such that $a_{i}=1$, for $i=i_{1},i_{2},\cdots,i_{N_1}$ and $a_{j}=0$ for $j\in\{i_{1},\cdots,i_{N_1}\}^c$, and
$$
 V_{\Z\alpha^{0}+\frac{1}{2}a_{0}\alpha^0}\otimes V_{\Z\alpha^1+\frac{1}{2}a_{1}\alpha^1}\otimes\cdots\otimes\cdots\otimes V_{\Z\alpha^{n}+\frac{1}{2}a_{n}\alpha^n}$$$$=\oplus_{0\leq p\leq n+1, p-N_1\in2\Z}L_{\widehat{{\frak g}}}(n+1,p)\otimes M^{{\delta'}}(p).
 $$
 Then
\begin{equation}\label{decom-5} M^1=M^{\underline{\delta}}(l_1)=\oplus_{0\leq q_1\leq n, q_1-|{\overline{\delta}'}|\in2\Z}M^{\overline{\delta'}}(q_1)\otimes L(c_{n}, h^{(n)}_{q_1+1,l_1+1}),
\end{equation}
\begin{equation}\label{decom-6}
M^2=M^{\overline{\delta}'}(l_2)=\oplus_{0\leq q_2\leq n,q_2-|\overline{\delta}'|\in2\Z}M^{\overline{\delta}'}(q_2)\otimes L(c_{n}, h^{(n)}_{q_2+1,l_2+1}),
\end{equation}
where $\overline{\delta}'=(a_{0},a_{1},\cdots,a_{n-1})$, $|\overline{\delta}'|=\sum\limits_{i=0}^{n-1}a_{i}$. Recall that
$$
M^{(n)}=M^{0^{n}}(0)=\oplus_{0\leq 2l\leq n}M^{0^{n-1}}(2l)\otimes L(c_{n}, h^{(n)}_{2l+1,1}).
$$
Let $N, N^{1}$, and $N^{2}$ be irreducible $U$-submodules of $M^{(n)}, M^{1}$ and $M^2$ respectively, where $U=M^{0^{n-1}}(0)\otimes L(c_{n}, 0)$ as above. Note that  $l_1\neq l_2$, $l_1-l_2\in 2\Z$ and in  (\ref{decom-5})-(\ref{decom-6}), $q_1-q_2\in 2\Z$. Then by Corollary \ref{wang-lam-co}
$$I_{L(c_{n},0)}\left(\begin{tabular}{c}
$L(c_{n},h^{(n)}_{q_1+1,l_1+1})$\\
$L(c_n, h^{(n)}_{2k+1,1})$ $L(c_n, h^{(n)}_{q_2+1,l_2+1})$\\
\end{tabular}\right)=0.$$
It follows from Theorem \ref{n2.2c} that
$$I_{U}\left(\begin{tabular}{c}
$N^1$\\
$N$\  $N^2$
\end{tabular}\right)=0.$$
By Lemma \ref{la1}, the lemma holds. \qed

By Lemmas \ref{l6.1}-\ref{l6.3} and Theorem \ref{n2.2}, we get our main result of this section.
\begin{theorem}\label{th6.1}
The simple vertex operator algebra $M^{(n)}$ is rational.
\end{theorem}

\vskip 0.5cm
{\bf Acknowledgements} \ The first author of the paper would like to thank Ching-hung Lam for valuable discussion and for telling her his paper in {\em Taiwanese Journal of Mathematics (2008)} jointly with Shinya Sakuma. The second author appreciate the hospitality from Shanghai Jiaotong University during his visit to Shanghai, when this work was carried out.


\begin{thebibliography}{99999}\frenchspacing
\label{referrence}
\bibitem[A]{A} T. Abe, Fusion rules for the charge conjugation
orbifold, {\em J. Alg.} {\bf 242} (2001), 624-655.
\bibitem[ADL]{ADL}T. Abe, C. Dong and H. Li, Fusion rules for the vertex
operator $M(1)^+$ and $V_L^+$, {\em Comm. Math. Phys.} {\bf 253}
(2005), 171-219.
\bibitem[ALY]{ALY} T. Arakawa, C. Lam and H. Yamada, Zhu's algebra, $C_2$-algebras and $C_2$-cofiniteness of parafermion vertex operator algebras,  arXiv: 1207.3909vl.
\bibitem[BEHHH]{BEHHH} R. Blumenhagen, W. Eholzer, A. Honecker, K. Hornfeck and R. H$\ddot{u}$bel, Coset realization of unifying $W$-algebras, {\em Internat. J. Modern Phys.} {\bf A10} (1995), 2367-2430.
\bibitem[B]{B}R. E. Borcherds, Vertex algebras, Kac-Moody algebras, and the Monster, {\it Proc. Natl. Acad. Sci. USA} {\bf 83} (1986), 3068-3071.
\bibitem[CL]{CL} T. Chen and C. Lam, Extension of the tensor product of unitary Virasoro vertex operator algebra, {\em Commun. Alg.} {\bf 35} (2007), 2487-2505.
\bibitem[D]{D} C. Dong, Representations of the moonshine modulevertex operator algebra, {\em Contemp. Math.} {\bf 175} (1994), 27-36.
\bibitem[DGH]{DGH} C. Dong, R. Griess Jr. and G. Hoehn, Framed vertex operator algebras, codes and the moonshine module,{\em Commun. Math. Phys.} {\bf 193} (1998), 407-448.
\bibitem[DJL]{DJL} C. Dong, C. Jiang and X. Lin, Rationality of vertex operator algebra $V_L^+:$ higher
rank, {\em Proc. Lond. Math. Soc.} {\bf 104} (2012), 799-826.
\bibitem[DLWY]{DLWY} C. Dong, C. Lam, Q. Wang and H. Yamada, The structure of parafermion vertex operator algebras, {\em J. Alg.} {\bf 323} (2010), 371-381.
\bibitem[DLY1]{DLY1} C. Dong, C. Lam and H. Yamada, Decomposition of the vertex operator algebra $V_{\sqrt{2}A_{3}}$, {\em J. Alg.}  {\bf 222} (1999), 500-510.
\bibitem[DLY2]{DLY2} C. Dong, C. Lam and H. Yamada, $W$-algebras related to parafermion vertex operator algebras, {\em J. Alg.} {\bf 322} (2009), 2366-2403.
\bibitem[DL]{DL}
C.~Dong and J.~Lepowsky, \textit{Generalized vertex algebras and relative vertex operators}, Progress in Math., Vol. {\bf112}, Birkh{\"a}user, Boston, 1993.
\bibitem[DLM1]{DLM1} C. Dong, H. Li and G. Mason, Regularity of rational vertex operator algebras, {\em Adv. Math.}{\bf 132} (1997), 148-166.
\bibitem[DLM2]{DLM2} C. Dong, H. Li and G. Manson, Simple currents and extensions of vertex operator algebras, {\em Commun. Math. Phys.} {\bf 180} (1996), 671-707.
\bibitem[DLM3]{DLM3} C. Dong, H. Li and G. Mason,  Twisted representations of
vertex operator algebras, {\em Math. Ann.} {\bf 310} (1998),
571-600.
\bibitem[DLM4]{DLM4} C. Dong, H. Li and G.  Mason,   Modular invariance of trace functions in orbifold theory and
generalized moonshine. {\em Comm. Math. Phys.}  {\bf 214} (2000), 1-56.
\bibitem[DLMN]{DLMN} C. Dong, H. Li, G. Mason and S. P. Notton, Associative subalgebras of the Griess algebra and related topics, in ''The Monster and Lie algebras'' (J. Ferrar and K. Harada, Eds.) de Gruyter, Berlin, 1998, pp. 27-42.
\bibitem[DMZ]{DMZ} C. Dong, G. Mason and Y. Zhu, Discrete series of the
Virasoro algebra and the moonshine module, {\em Proc. Symp. Pure.
Math. American Math. Soc.} {\bf 56} II (1994), 295-316.
\bibitem[DW1]{DW1} C. Dong and Q. Wang, The structure of parafermion vertex operator algebras: general case, {\em Commun. Math. Phys.} {\bf 299} (2010), 783-792.
\bibitem[DW2]{DW2} C. Dong and Q. Wang, On $C_{2}$-cofiniteness  of parafermion vertex operator algebras,  {\em J. Alg.} {\bf 328} (2011), 420-431.
\bibitem[FF]{FF} B. Feigin and D. Fuchs, Verma modules over the Virasoro algebra, Topology (Leningrad, 1982), {\em Lecture Notes in Math.} 1060. Berlin-New York, Springer: 230-245.
\bibitem[FHL]{FHL} I. B. Frenkel, Y. Huang and J. Lepowsky, On
axiomatic approaches to vertex operator algebras and modules, {\it
Memoirs American Math. Soc.} {\bf 104}, 1993.
\bibitem[FLM]{FLM} I. B. Frenkel, J. Lepowsky and A. Meurman,Vertex Operator Algebras and the Monster, {\em Pure and Applied Math.} Vol. {\bf 134}, Academic Press, 1988.
\bibitem[FZ]{FZ} I. B. Frenkel and Y. Zhu, Vertex operator algebras associated to representations of affine and Virasoro algebras, {\em Duke Math. J. } {\bf 66} (1992), 123-168.
\bibitem[GZ]{GZ} D. Gepner and Z. Qiu, Modular invariant partition functions for parafermionic field theories, {\em Nucl.
Phys.} B285 (1987), 423-453.
\bibitem[GKO1]{GKO1} P. Goddard, A. Kent and D. Olive, Virasoro algebras and coset space models, {\em Phys. Lett.} {\bf B152} (1985), 88-92.
\bibitem[GKO2]{GKO2} P. Goddard, A. Kent and D. Olive, Unitary representations of the Virasoro and super-Virasoro algebra, {\em Commun. Math. Phys.} {\bf 103} (1986), 105-119.
\bibitem[J]{Ja} G. James, The representation theory of the symmetric groups, Springer-Verlag  Berlin Heidelberg New York 1978.
\bibitem [K]{K}
V. G. Kac, {\it Infinite-dimensional Lie Algebras}, 3rd ed.,
Cambridge Univ. Press, Cambridge, 1990.

\bibitem[KR]{KR} V. G. Kac and A. Raina, Highest Weight Representations of Infinite
Dimensional Lie Algebras, World Scientific, Adv. Ser. In Math.
Phys., Singapore, 1987.
\bibitem[KW]{KW} V. Kac and W. Wang, Vertex operator superalgebras and their representations, {\em Contemp. Math. Amer. Math. Soc.} {\bf 175} (1994), 161-191.
\bibitem[KL]{KL} Y. Kawahigashi and R. Longo, Classification of local conformal nets. Case $c < 1$, {\em Ann. of Math.} {\bf 160} (2004), 493¨C522.
\bibitem[KLY]{KLY} Kitazume, C. Lam and H. Yamada, Decomposition of the moonshine vertex operator algebra as Virasoro modules, {\em J. Alg.} {\bf 226} (2000), 893-919.
\bibitem[KMY]{KMY}  Kitazume, M. Miyamoto and H. Yamada, Ternary codes and vertex operator algebras, {\em J. Alg.} {\bf 223} (2000), 379-395.
\bibitem[LL]{LL} J. Lepowsky and H. Li, Introduction to Vertex Operator Algebras and Their Representations,{\em  Progress in Mathematics,} Vol. {\bf  227}, Birkh äuser Boston, Inc., Boston, MA, 2004.
 \bibitem[L]{L} H. Li, Some finiteness properties of regular vertex
operator algebras, {\em J. Alg.} \textbf{212} (1999) 495-514.
\bibitem [LS]{LS} C. Lam and S. Sakuma, On a class of vertex operator algebras having a faithful $S_{n+1}$-action, {\em Taiwanese Journal of Mathematics}, {\bf 12} (2008), 2465-2488.
\bibitem [LY]{LY} C. Lam and H. Yamada, Decomposition of the lattice vertex operator algebra $V_{\sqrt{2}A_{l}}$, {\em J. Alg.} {\bf 272} (2004), 614-624.
\bibitem[MNT]{MNT} A. Matsuo, K. Nagatomo and  A. Tsuchiya, Quasi-finite algebras graded by Hamiltonian and vertex operator algebras, math. QA/0505071.
\bibitem[M]{M} M. Miyamoto, Representation theory of code vertex operatoralgebra, {\em  J. Algebra} {\bf  201} (1998), 115--150.
\bibitem[MT]{MT} M. Miyamoto and K. Tanabe, Uniform product of $A\sb {g,n}(V)$ for an orbifold model $V$ and $G$-twisted Zhu algebra,  {\em J. Alg.} {\bf  274} (2004), 80-96.
\bibitem[X1]{X1} F. Xu, Algebraic coset conformal field theories II, {\em Publ. Res. Inst. Math. Sci.} {\bf 35} (1999), 795¨C824.
\bibitem[X2]{X2} F. Xu, Algebraic coset conformal field theories, {\em Commnu. Math. Phys.}  {\bf 211} (2000), 1-43.
\bibitem[ZF]{ZF}  A. B. Zamolodchikov and V. A. Fateev, Nonlocal (parafermion) currents in two-dimensional conformal
quantum field theory and self-dual critical points in ZN-symmetric statistical systems, {\em Sov.
Phys. JETP} {\bf 62} (1985), 215-225.
\bibitem[W]{W} W. Wang, Rationality of Virasoro vertex operator algebras,
{\em Internat. Math. Res. Notices,} {\bf  7} (1993), 197-211.
\bibitem[Z]{Z} Y. Zhu, Modular invariance of characters of vertex operator algebras, {\em J. Amer, Math. Soc.} {\bf 9} (1996), 237-302.

\end{thebibliography}
\end{document}